\documentclass{CSML}
\pdfoutput=1

\usepackage{lastpage}
\newcommand{\dOi}{13(3:4)2017}
\lmcsheading{}{1--\pageref{LastPage}}{}{}%
{Jan.~16, 2017}{Jul.~10, 2017}{}

\usepackage[utf8]{inputenc}
\usepackage[T1]{fontenc}
\usepackage{lmodern}

\usepackage[leqno]{amsmath}
\usepackage{amssymb,amsthm}
\usepackage[mathscr]{euscript}
\usepackage{bbm}
\usepackage{dsfont}
\usepackage{stmaryrd}
\usepackage{enumitem}
\usepackage[all]{xy}
\usepackage{mathtools} 

\theoremstyle{plain}
\newtheorem{theorem}{Theorem}[section]
\newtheorem{lemma}[theorem]{Lemma}
\newtheorem{proposition}[theorem]{Proposition}
\newtheorem{corollary}[theorem]{Corollary}

\theoremstyle{definition}
\newtheorem{definition}[theorem]{Definition}
\newtheorem{examples}[theorem]{Examples}
\newtheorem{example}[theorem]{Example}
\newtheorem{assumption}[theorem]{Assumption}

\theoremstyle{remark}
\newtheorem{remark}[theorem]{Remark}

\newtheorem*{notation}{Notation}

\newlist{tfae}{enumerate}{1}
\setlist[tfae,1]{label=(\roman*)}

\usepackage{chngcntr}

\newcommand{\ff}{\mathfrak{f}}

\newcommand{\calM}{\mathcal{M}}

\newcommand{\fF}{\mathfrak{F}}

\DeclareMathOperator{\upc}{\mathord{\uparrow}}

\newcommand{\catfont}[1]{\mathsf{#1}}

\newcommand{\catX}{\catfont{X}}
\newcommand{\catA}{\catfont{A}}
\newcommand{\catB}{\catfont{B}}
\newcommand{\catC}{\catfont{C}}
\newcommand{\catD}{\catfont{D}}

\newcommand{\SET}{\catfont{Set}}

\newcommand{\POSS}{\catfont{Pos}}

\newcommand{\SUP}{\catfont{Sup}}
\newcommand{\DSUP}{\catfont{DSup}}
\newcommand{\TOP}{\catfont{Top}}
\newcommand{\STCOMP}{\catfont{StablyComp}}
\newcommand{\SPEC}{\catfont{Spec}}
\newcommand{\CONTLAT}{\catfont{ContLat}}

\newcommand{\ALat}{\catfont{ALat}}
\newcommand{\ADom}{\catfont{ADom}}

\newcommand{\op}{\mathrm{op}}
\newcommand{\alg}{\mathrm{alg}}

\DeclareMathOperator{\Spl}{Spl}
\DeclareMathOperator{\kar}{kar}

\DeclareMathOperator{\omono}{ord-mono}
\DeclareMathOperator{\oepi}{ord-epi}

\DeclareMathOperator{\id}{id}

\newcommand{\longto}{\longrightarrow}

\newcommand{\modto}{\mathrel{\mathmakebox[\widthof{$\xrightarrow{\rule{1.45ex}{0ex}}$}]
{\xrightarrow{\rule{1.45ex}{0ex}}\hspace*{-2.8ex}{\circ}\hspace*{1ex}}}} 

\newcommand\adjunct[2]{\xymatrix@=8ex{\ar@{}[r]|{\top}\ar@<1mm>@/^2mm/[r]^{{#2}} & \ar@<1mm>@/^2mm/[l]^{{#1}}}}
\newcommand\adjunctop[2]{\xymatrix@=8ex{\ar@{}[r]|{\bot}\ar@<1mm>@/^2mm/[r]^{{#2}} & \ar@<1mm>@/^2mm/[l]^{{#1}}}}

\newcommand{\monadfont}[1]{\mathbbm{#1}}

\newcommand{\mT}{\monadfont{T}}
\newcommand{\mD}{\monadfont{D}}
\newcommand{\mI}{\monadfont{I}}
\newcommand{\mF}{\monadfont{F}}
\newcommand{\mFO}{\monadfont{F}_1}
\newcommand{\mFp}{\monadfont{F}_2}
\newcommand{\mFc}{\monadfont{F}_c}

\newcommand{\monad}{(T,m,e)}
\newcommand{\dmonad}{(D,m,e)}
\newcommand{\imonad}{(I,m,e)}
\newcommand{\fmonad}{(F,m,e)}
\newcommand{\fOmonad}{(F_1,m,e)}
\newcommand{\fpmonad}{(F_2,m,e)}
\newcommand{\fcmonad}{(F_c,m,e)}

\newcommand{\xupt}{\catX^{\mT}}
\newcommand{\xdot}{\catX_{\mT}}

\DeclareMathOperator{\Id}{Id}

\newcommand{\eps}{\varepsilon}

\newcommand{\df}[1]{\emph{\textbf{#1}}}

\newcommand{\To}{\longrightarrow}

\newcommand{\medvee}{\mbox{\Large $\vee$}}

\newcommand{\medcap}{\mathbin{\mbox{\Large $\cap$}}}

\newcommand{\s}{\mathrm{S}}

\newcommand{\Hhom}{\mathbb{H}\mathrm{om}}

\newcommand{\cote}{\mathord\pitchfork}

\newcommand{\goo}{\Omega}
\newcommand{\ra}{\rightarrow}

\newcommand{\ga}{\alpha}

\newcommand{\vdir}{\bigvee^{\upc}}

\newcommand{\KZ}{Kock-Z\"oberlein}

\title{Aspects of algebraic Algebras}

\author{Dirk Hofmann}

\address{CIDMA, Department of Mathematics,
  University of Aveiro, Portugal}

\email{dirk@ua.pt}

\author{Lurdes Sousa}

\address{CMUC, University of Coimbra and Polytechnic Institute of Viseu,
  Portugal}

\email{sousa@estv.ipv.pt}

\thanks{Research partially supported by Centro de Investiga\c{c}\~ao e
  Desenvolvimento em Matem\'atica e Aplica\c{c}\~oes da Universidade
  de Aveiro/FCT, and by Centro de Matem\'{a}tica da Universidade de
  Coimbra -- UID/MAT/00324/2013 -- UID/MAT/04106/2013, funded by the
  Portuguese Government through FCT/MEC and co-funded by the European
  Regional Development Fund through the Partnership Agreement PT2020}



\usepackage[hypertexnames=false]{hyperref}
\hypersetup{hidelinks} 

\begin{document}

\begin{abstract}
  In this paper we investigate important categories lying strictly between the
  Kleisli category and the Eilenberg--Moore category, for a Kock-Z\"oberlein
  monad on an order-enriched category. Firstly, we give a characterisation of
  free algebras in the spirit of domain theory. Secondly, we study the existence
  of weighted (co)limits, both on the abstract level and for specific categories
  of domain theory like the category of algebraic lattices. Finally, we apply
  these results to give a description of the idempotent split completion of the
  Kleisli category of the filter monad on the category of topological spaces.
\end{abstract}

\dedicatory{We dedicate this paper to Jirka Adámek whose mathematics has
  enchanted us since the first seminars\\ in Bremen and in Coimbra.}

\maketitle

\section{Introduction}
\label{sec:introduction}

The Eilenberg-Moore categories of idempotent monads are precisely the full
reflective isomorphism-closed subcategories of the base category. A substantial
study in category theory has been dedicated to full reflective subcategories
since the 1970's, and this is one of the many subjects to which Ji\v{r}\'{i}
Ad\'amek has given a remarkable contribution (see \cite{AK81,AR88,ART88}, just
to name a few).  The notion of {\KZ} monad (\cite{Koc95,Zob76}), also named
lax-idempotent monad, is a fruitful generalisation of idempotent monads to the
more general setting of 2-categories. In particular, it provides a new insight
into important examples of domain theory and topology, when our 2-categories are
just order-enriched categories. On this subject, we refer to a series of papers
in the late 1990's by M.\ Escard{\'o} and others (e.g.,
\cite{Esc98,EF99,Esc03}). In this case, the Eilenberg-Moore categories are
reflective subcategories of the base category as well; however, in general they
are not anymore full. In \cite{ASV15} and other related papers, this kind of
subcategories were called \emph{KZ-monadic subcategories}. As demonstrated in a
series of recent papers \cite{CS11,ASV15,Sou17,AS17}, several important
well-known properties and notions on full reflective subcategories of ordinary
category theory have an order-enriched counterpart when we replace full
reflectivity by KZ-monadicity.

Associated with each monad $\mT=\monad$ on a category $\catX$, we have a
faithfully full functor $E:\xdot {\hookrightarrow} \xupt$ between the Kleisli
category $\xdot$ and the Eilenberg-Moore category $\xupt$.  Moreover, we have
adjunctions $F_{\mT}\dashv U_{\mT}: \xdot \To \catX$ and
$F^{\mT}\dashv U^{\mT}: \xupt \To \catX$ with $U^{\mT}E=U_{\mT}$ and
$EF_{\mT}=F^{\mT}$. In fact, these two adjunctions are the initial and terminal
objects of the obvious category of all adjunctions which induce the monad $\mT$.

When $\mT$ is an idempotent monad (i.e., the multiplication $m$ is a natural
isomorphism), $\xupt$ can be identified with a full subcategory of $\catX$, and
$\xdot\simeq \xupt$. Thus, when the functor $T$ is injective on objects, as it
happens in most significant examples, $\xupt$ is just the closure under
isomorphisms of $\xdot$ in $\catX$. Hence, there are no interesting
subcategories strictly between $\xdot$ and $\xupt$ to be considered.

The situation is dramatically different when we work with Kock-Z\"{o}berlein
monads in order-enriched categories. In this case we have yet $\xupt$ as a
(usually non-full) subcategory of $\catX$. And $\xupt$ is now, on objects and
morphisms, the closure of $\xdot$ under left adjoint retractions on $\catX$
(\cite{CS11}). But, between $\xdot$ and $\xupt$ there are interesting
subcategories which are quite distinct. As an example, take the open filter
monad $\mF$ over $\catX=\TOP_0$. Then $\catX^{\mF}$ is precisely the category of
continuous lattices and maps preserving directed suprema and arbitrary
infima. And between $\catX_{\mF}$ and $\catX^{\mF}$ we have at least two
remarkable subcategories: the category $\ALat$ of algebraic lattices is properly
contained in $\catX^{\mF}$, and between $\catX_{\mF}$ and $\ALat$ we have the
idempotent completion of $\catX_{\mF}$ which we characterise here as consisting
of all algebraic lattices whose compact elements form the dual of a frame
(Section~\ref{sec:idemp-split-compl}). We also prove that, for that monad $\mF$,
$\ALat$ is precisely the closure under weighted limits of $\catX_{\mF}$ in
$\catX^{\mF}$ (then, also in $\TOP_0$).

In this paper we embark on a study of important categories lying strictly
between the Kleisli category and the Eilenberg--Moore category, for a
Kock-Z\"oberlein monad on an order-enriched category; with particular focus on
various filter monads on the category $\TOP_0$ of T$_0$ topological spaces and
continuous maps. After recalling the necessary background material in
Section~\ref{sec:backgr-mater-kock}, the aim of
Section~\ref{sec:abstr-algebr-objects} is to give a general treatment of the
notion of algebraic lattice. In continuation of \cite{RW04}, where the authors
observe that ``these theorems characterizing completely distributive lattices
are not really about lattices'' but rather ``about a mere monad $D$ on a mere
category'', in Theorem~\ref{thm:char_algebraic} we give a characterisation of
free algebras for a general Kock-Z\"oberlein monad, the \emph{algebraic
  algebras}, which resembles the classical notion of (totally) algebraic
lattice.

Taking seriously the fact that $\TOP_0$ is order-enriched forces us to not just
consider ordinary completeness but rather study \emph{weighted} limits and
colimits. In this spirit, in Section~\ref{sec:weight-colim-cogen} we prove an
interesting general result which has an important application in
Section~\ref{sec:cocompl-subc-xupt}: every order-enriched category with weighted
limits and a regular cogenerator has also weighted colimits.

In Section~\ref{sec:cocompl-subc-xupt} we consider the full subcategory of
$\catX^\mT$ defined by those algebras which are in a suitable sense
\emph{cogenerated} by the Sierpi\'nski space. For the various filter monads we
show that these algebras coincide with well-known objects in domain theory:
algebraic lattices and spectral spaces. In particular, we conclude that the
corresponding categories have weighted limits and weighted colimits.

Finally, in Section~\ref{sec:idemp-split-compl}, we consider the filter monad
$\mF$ on $\TOP_0$. By the results of the previous section, its Kleisli category
is a full subcategory of the category $\ALat$ of algebraic lattices with maps
preserving directed suprema and all infima. As the latter one is complete, it
contains in particular the idempotent split completion of $(\TOP_0)_\mF$; and we
identify the objects of this completion as precisely those algebraic lattices
where the compact elements form the dual of a frame.

\section{Background material on Kock-Z\"oberlein monads}
\label{sec:backgr-mater-kock}

In this section we recall the main facts about Kock-Z\"oberlein monads on
order-enriched categories needed in this paper. For general 2-categories, this
type of monads were introduced independently by Volker Z\"oberlein \cite{Zob76}
and Anders Kock (see \cite{Koc95}). We also refer to \cite{Esc98} and
\cite{EF99} for a detailed study of Kock-Z\"oberlein monads in the context of
domain theory, the one treated in this paper. In particular, the three theorems
of this section are presented there (see also \cite{Koc95}).  In this case, we
work with a special type of 2-categories, the \df{order-enriched} categories,
that is, categories enriched in the category $\POSS$ of partially ordered sets
and monotone maps. This means that the hom-sets are posets and the composition
of morphisms preserves the order on the left and on the right. An order-enriched
functor between order-enriched categories is one which preserves the order of
each hom-poset.

\begin{definition}
  A monad $\mT=\monad$ on an order-enriched category $\catX$ is called
  \df{order-enriched} whenever $T:\catX \to \catX$ is so. An order-enriched
  monad $\mT=\monad$ is \df{of Kock-Z\"oberlein type} whenever $Te_X\le e_{TX}$,
  for all object $X$ in $\catX$.
\end{definition}

We note that, for an order-enriched monad $\mT=\monad$, the full and faithful
functor $\catX_\mT\to\catX^\mT$ of the Kleisli category into the Eilenberg-Moore
category is also an order-isomorphism on hom-sets. The condition
``$Te_X\le e_{TX}$'' in the definition of Kock-Z\"oberlein monad is somehow
arbitrarily chosen, the following theorem (see \cite{Koc95}) presents an
alternative descriptions.

\begin{theorem}\label{thm:char-KZ-monad}
  Let $\mT=\monad$ be an order-enriched monad on an order-enriched category
  $\catX$. For every object $X$ in $\catX$, the following assertions are
  equivalent.
  \begin{tfae}
  \item $Te_X\le e_{TX}$.
  \item $m_X\dashv e_{TX}$.
  \item $Te_X\dashv m_X$.
  \end{tfae}
\end{theorem}

We turn now our attention to Eilenberg--Moore algebra structures, which can be
characterised using adjunction (see \cite{Koc95}).

\begin{theorem}\label{d:thm:1}
  Let $\mT=\monad$ be a Kock-Z\"oberlein monad on an order-enriched category
  $\catX$ and let $\alpha:TX\to X$ in $\catX$. Then the following assertions are
  equivalent.
  \begin{tfae}
  \item $\alpha:TX\to X$ is a $\mT$-algebra structure on $X$.
  \item $\alpha\cdot e_X=\id_X$.
  \item $\alpha\dashv e_X$.
  \end{tfae}
\end{theorem}

As a consequence of the equivalence (i) $\iff$ (iii) of the above theorem, the
Eilenberg-Moore category $\xupt$ is a subcategory of $\catX$ (up to isomorphism
of categories). Moreover, $\catX^\mT$ is also an order-enriched category with
the order inherited from $\catX$.

\begin{remark}\label{d:rem:1}
  Let $f:X\to Y$ be a left adjoint in $\catX$ between $\mT$-algebras, with
  $f\dashv g$. Then, using Theorem~\ref{d:thm:1}, we have that the equality
  $Tg\cdot e_Y=e_X\cdot g$ implies $\beta\cdot Tf=f\cdot \alpha$ by unicity of
  adjoints. Consequently, every left adjoint between $\mT$-algebras is a
  homomorphism. In the general setting of 2-categories, this is Proposition~2.5
  of \cite{Koc95}.
\end{remark}

Before presenting examples, we recall some standard notions from order theory
and topology.
\begin{definition}\label{def:domains}
  In (1)-(6) we follow the terminology of \cite{GHK+03}.
  \begin{enumerate}
  \item A subset $D\subseteq X$ of a partially ordered set $X$ is called
    \df{directed} whenever $D\neq\varnothing$ and, for all $x,y\in D$, there is
    some $z\in D$ with $x\le z$ and $y\le z$. We are going to use the notation
    $\bigvee^{\upc}D$ to express the supremum of a set $D$ and, at the same
    time, indicate that $D$ is directed.
  \item The \df{way below} relation $\ll$ is defined as follows: $x\ll y$
    provided that, for every directed subset $D\subseteq X$, if
    $y\leq \bigvee^{\upc} D$, then $x\leq d$ for some $d\in D$. An element
    $x\in X$ is called \df{compact} whenever $x\ll x$.
  \item The \df{totally below} relation $\lll$ is defined in a similar way:
    $x\lll y$ whenever, for every subset $S\subseteq X$, if $y\leq \bigvee S$,
    then $x\leq d$ for some $d\in S$. An element $x\in X$ is called \df{totally
      compact} whenever $x\lll x$.
  \item A partially ordered set $X$ is called \df{directed complete} whenever
    every directed subset of $X$ has a supremum. Furthermore, $X$ is said to be
    \df{bounded complete} if every subset with an upper bound has a least one;
    equivalently, it has all non-empty infima.
  \item A partially ordered set $X$ is \df{continuous} if each one of its
    elements $x$ is the directed supremum of all elements $y$ with $y\ll x$. A
    \df{domain} is a continuous poset with directed suprema.  Furthermore, a
    complete partially ordered set $X$ is called \df{completely distributive}
    whenever every $x\in X$ is the supremum of all elements $y$ with $y\lll x$.
  \item\label{item:algebraic} A domain $X$ with each $x\in X$ satisfying the
    equality $x=\bigvee^{\upc}\{y\in X\mid y\leq x, \, y\ll y\}$ is an
    \df{algebraic domain}. The designation of \df{continuous lattice}
    \cite{Sco72} is used for a domain which is also a lattice; hence, a
    continuous lattice is a complete and continuous partially ordered
    set. Analogously, an \df{algebraic lattice} is an algebraic domain which is
    also a lattice. A completely distributive partially ordered set where
    $x=\bigvee\{y\in X\mid y\leq x, \, y\lll y\}$, for every $x\in X$, is called
    \df{totally algebraic}.
  \item\label{item:stablycompact} A topological space $X$ is called \df{stably
      compact} whenever $X$ is sober, locally compact and every finite
    intersection of compact saturated subsets is compact (where to be saturated
    means to be an upper subset with respect to the specialisation order, see
    \cite{Jun04}). A continuous map $f:X\to Y$ between stably compact spaces is
    called \df{spectral} whenever $f^{-1}(K)$ is compact, for every compact
    saturated subset $K\subseteq Y$. We denote by $\STCOMP$ the category of
    stably compact spaces and spectral maps. A stably compact space $X$ is
    called \df{spectral} whenever the compact open subsets form a basis for the
    topology of $X$; equivalently, if the cone $(f:X\to\s)_f$ of all spectral
    maps into the the Sierpi\'nski space is initial with respect to the
    forgetful functor $\TOP\to\SET$; and this in turn is equivalent to
    $(f:X\to\s)_f$ being initial with respect to the canonical forgetful functor
    $\STCOMP\to\SET$. It is also well-known that a continuous map $f:X\to Y$
    between spectral spaces is spectral if and only if $f^{-1}(K)$ is compact,
    for every compact open subset $K\subseteq Y$. The full subcategory of
    $\STCOMP$ defined by all spectral spaces we denote by $\SPEC$; it is a
    reflective subcategory since by definition it is closed under initial cones
    (see \cite[Theorem 16.8]{AHS90}). Finally, we note that $\STCOMP$ is
    equivalent to the category of Nachbin's partially ordered compact Hausdorff
    spaces and monotone continuous maps. Here a stably compact space $X$
    corresponds to the partially ordered compact Hausdorff space with the same
    underlying set, the order relation is the specialisation order, and the
    compact Hausdorff topology is given by the so-called patch topology (see
    \cite{Nac65,Jun04} for details).
  \end{enumerate}
\end{definition}

\begin{examples}\label{exs:1}
  The following monads are of Kock-Z\"oberlein type.
  \begin{enumerate}
  \item The category $\POSS$ of partially ordered sets and monotone maps is
    order-enriched, with the pointwise order of monotone maps. The downset monad
    $\mD=\dmonad$ on $\POSS$ is given by
    \begin{itemize}
    \item the downset functor $D:\POSS\to\POSS$ which sends an ordered set $X$
      to the set $DX$ of downclosed subsets of $X$ ordered by inclusion, and,
      for $f:X\to Y$ monotone, $Df:DX\to DY$ sends a downclosed subset $A$ of
      $X$ to the downclosure of $f(A)$;
    \item the unit $e_X:X\to DX$ sends $x\in X$ to the downclosure
      $\downarrow x$ of $x$; and
    \item the multiplication $m_X:DDX\to DX$ sends a downset of downsets to its
      union.
    \end{itemize}
    The category $\POSS^\mD$ of Eilenberg--Moore algebras and homomorphisms is
    equivalent to the category $\SUP$ of complete partially ordered sets and
    sup-preserving maps.
  \item An interesting submonad of $\mD=\dmonad$ is given by the monad
    $\mI=\imonad$ where $IX$ is the set of directed downclosed subset of $X$,
    ordered by inclusion. Furthermore, $\POSS^\mI$ is equivalent to the category
    $\DSUP$ of partially ordered sets with directed suprema and maps preserving
    directed suprema.
  \item We denote the category of topological T$_0$-spaces and continuous maps
    by $\TOP_0$. The topology of a T$_0$-space $X$ induces the specialisation
    order on the set $X$: for $x,x'\in X$,
    $x \le x' \iff \goo(x)\subseteq\goo(x')$, where $\goo(x)$ denotes the set of
    open sets. Every continuous map preserves this order, and, thus, also its
    dual. We consider $\TOP_0$ as an order-enriched category by taking the dual
    of the specialisation order pointwisely on hom-sets.

    The filter functor $F:\TOP_0\to\TOP_0$ sends a topological space $X$ to the
    space $FX$ of all filters on the lattice $\goo X$ of open subsets of
    $X$. The topology on $FX$ is generated by the sets
    \begin{align*}
      A^\#&=\{\ff\in FX\mid A\in\ff\}
    \end{align*}
    where $A\subseteq X$ is open. For a continuous map $f:X\to Y$, the map
    $Ff:FX\to FY$ is defined by
    \[
      \ff\mapsto\{B\in \Omega Y\mid f^{-1}(B)\in\ff\},
    \]
    for $\ff\in FX$. Since $(Ff)^{-1}(B^\#)=(f^{-1}(B))^\#$ for every
    $B\subseteq Y$ open, $Ff$ is continuous. The filter functor is part of the
    filter monad $\mF=\fmonad$ on $\TOP_0$, here the unit $e_X:X\to FX$ sends
    $x\in X$ to its neighbourhood filter $\goo (x)$, and the multiplication
    $m_X:FFX\to FX$ sends $\fF\in FFX$ to the filter
    $\{A\subseteq X\mid A^\#\in\fF\}$. The category $\TOP_0^\mF$ of
    Eilenberg--Moore algebras for the filter monad is equivalent to the category
    $\CONTLAT$ of continuous lattices and maps preserving directed suprema and
    arbitrary infima (see \cite{Day75,Wyl81}). Here a continuous lattice is
    viewed as a topological space with the Scott topology, and the algebra
    structure $\alpha:FX\to X$ picks for every $\ff\in FX$ the largest
    convergence point with respect to the specialisation order.
  \item\label{item:StablyComp} In this paper we will consider several submonads
    of the filter monad $\mF$ on $\TOP_0$; in particular, the proper filter
    monad $\mFO=\fOmonad$ where $F_1X$ is the subspace of $FX$ consisting of all
    proper filters, and the prime filter monad $\mFp=\fpmonad$ where $F_2X$ is
    the subspace of $FX$ consisting of all prime filters.  Indeed, we have a
    chain of Kock-Z\"oberlein submonads $\mF_n$ of $\mF$, for $n$ a regular
    cardinal, where $F_nX$ is the subspace of $FX$ of all $n$-prime filters;
    that is, filters with the property that, for each union of an $n$-indexed
    family of open sets belonging to the filter, some member of the family
    belongs to the filter too. The union of this chain is the completely prime
    filter Kock-Z\"oberlein monad $\mFc=\fcmonad$ where $F_cX$ is the subspace
    of $FX$ consisting of all completely prime filters (see \cite{CS15}). In the
    latter case, the category $\TOP_0^{\mFc}$ is equivalent to the category of
    sober spaces and continuous maps (see \cite{EF99}). It is shown in
    \cite{Sim82} that the category $\TOP_0^{\mFp}$ is equivalent to the category
    $\STCOMP$ of stably compact spaces and spectral maps. Moreover,
    $\TOP_0^{\mFO}$ is equivalent to the category of bounded complete domains
    (also known as continuous Scott domains) and maps preserving directed
    suprema and non-empty infima (see \cite{Wyl85,EF99}).
  \end{enumerate}
\end{examples}

The notion of Kock-Z\"oberlein monad generalises the one of idempotent monad; we
recall that a monad $\mT=\monad$ on a category $\catX$ is idempotent whenever
$m:TT\to T$ is an isomorphism. By Theorem~\ref{thm:char-KZ-monad}, $\mT$ is
idempotent if and only if $\mT$ is of Kock-Z\"oberlein type with respect to the
discrete order on the hom-sets of $\catX$; i.e.\ if $Te_X=e_{TX}$. This
observation motivates the designation \emph{lax idempotent monad} for this type
of monads, which is also used in the literature. Furthermore, we recall that
\vspace{1em}
\begin{enumerate}
\item For every adjunction $\smash{\catA\adjunct{F}{G}\catX}$, $G$ is fully faithful if
  and only if the counit $\eps:FG\to\Id$ is an isomorphism.
\item Every fully faithful and right adjoint functor $G:\catA\to\catX$ is
  monadic, and the induced monad is idempotent.
\item For every monad $\mT$ on $\catX$, $G^\mT:\catX^\mT\to\catX$ is full if and
  only if $\mT$ is idempotent.
\end{enumerate}
We also remark that the completely prime filter monad $\mFc=\fcmonad$ on
$\TOP_0$ is actually idempotent.

For an order-enriched monad $\mT=\monad$ on an order-enriched category $\catX$,
we put
\[
  \calM_\mT=\{h:X\to Y\text{ in }\catX\mid \text{$Th$ has a right adjoint
    $g:TY\to TX$ satisfying }g\cdot Th=\id_{TX}\}.
\]
Clearly, if $\catX$ is locally discrete, then $\calM_\mT$ is the class of all
morphisms $h:X\to Y$ where $Th$ is an isomorphism. Equivalently, $\calM_\mT$ is
the largest class of morphisms of $\catX$ with respect to which the subcategory
$\catX^{\mT}$ is orthogonal. The concept of orthogonality is the
particularisation, to the locally discrete case, of the concept of
\emph{Kan-injectivity}. We recall that an object $A$ is left Kan-injective with
respect to a morphism $h:X\to Y$, if and only if the hom-map
$\,\catX(h,A):\catX(Y,A)\ra \catX(X,A)\,$ is a right adjoint retraction in the
category $\POSS$. And a morphism $f:A\ra B$ is Kan-injective with respect to $h$
if $A$ and $B$ are so and the left adjoint maps $(\catX(h,A))^{\star}$ and
$(\catX(h,B))^{\star}$ satisfy the equality
$\catX(X,f)\cdot (\catX(h,A))^{\star}=(\catX(h,B))^{\star}\cdot \catX(Y,f)$.

Next we recall a characterisation of Eilenberg--Moore algebras of
Kock-Z\"oberlein monads in terms of injectivity.

\begin{theorem}\label{thm:algebra_vs_injective}
  Let $A$ be in $\catX$ and $\mT$ be a Kock-Z\"oberlein monad on $\catX$. Then
  the following assertions are equivalent.
  \begin{tfae}
  \item $A$ is injective with respect to
    $\{e_X:X\to TX\mid X\text{ in }\catX\}$.
  \item $A$ is a $\mT$-algebra.
  \item $A$ is injective with respect to $\calM_\mT$.
  \item $A$ is Kan-injective with respect to $\calM_\mT$.
  \end{tfae}
\end{theorem}

Moreover, as shown in \cite{CS11}, $\calM_\mT$ is the largest class of morphisms
of $\catX$ with respect to which the subcategory $\catX^{\mT}$ is
Kan-injective. For a detailed study on Kan-injectivity, see also \cite{ASV15}.

If $\mT$ is of Kock-Z\"oberlein type, then, by Theorem~\ref{thm:char-KZ-monad},
$e_X:X\to TX$ belongs to $\calM_\mT$, for all objects $X$ in $\catX$. However,
in contrast to the idempotent case, the following example shows that this
property does not characterise Kock-Z\"oberlein monads.

\begin{example}
  Let $\catX$ be the order-enriched category of all complete partially ordered
  sets and all monotone maps, ordered pointwise; and let $\catA$ be the
  subcategory of $\catX$ with the same objects, and as morphisms those morphisms
  of $\catX$ which preserve the top and the bottom element. The inclusion
  functor $\catA\hookrightarrow\catX$ is right adjoint: for each object $X$ of
  $\catX$, the reflection map
  \[
    \eta_X:X\to FX=\{\bot\}+X+\{\top\}
  \]
  is given by freely adjoining a largest and a smallest element to
  $X$. Furthermore, $F\eta_X:FX\to FFX$ sends the bottom element of $FX$ to the
  bottom element of $FFX$, the top element of $FX$ to the top element of $FFX$,
  and $x\in X$ to itself. Since $X$ has a largest element, every supremum in
  $FX$ of elements of $X$ is in $X$, therefore $F\eta_X$ preserves all suprema
  and consequently has a right adjoint $g:FFX\to FX$ in $\catX$. Moreover, since
  $F\eta_X:FX\to FFX$ is an order-embedding, we obtain
  $g\cdot F\eta_X=\id_{FX}$. Since $\eta_{FX}$ neither preserves the top nor the
  bottom element, we get $F\eta_X\nleq\eta_{FX}$ and $\eta_{FX}\nleq F\eta_X$;
  in particular, the induced monad is not of Kock-Z\"oberlein type, neither for
  the order $\le$ nor for its dual.
\end{example}

\begin{remark}\label{rem:1}
  For every $h:X\to Y$ in $\catX$, a right adjoint $g:TY\to TX$ of $Th$ is
  necessarily a $\mT$-algebra homomorphism. To see this, just observe that the
  diagram
  \[
    \xymatrix{ TTX\ar[r]^{TTh} & TTY\\ TX\ar[u]^{Te_X}\ar[r]_{Th} &
      TY\ar[u]_{Te_Y}}
  \]
  commutes, therefore the diagram of the corresponding right adjoints
  $Th\dashv g$, $TTh\dashv Tg$, $Te_X\dashv m_X$ and $Te_Y\dashv m_Y$ commutes
  as well. We also recall that, for an adjunction $f\dashv g$ in an
  order-enriched category, the inequalities $\id\le gf$ and $fg\le\id$ imply
  $fgf=f$; hence, if $f$ is a monomorphism, then $gf=\id$. Consequently,
  \begin{align*}\calM_\mT&=\{h:X\to Y\text{ in }\catX\mid \text{$Th$
                           is a left adjoint
                           monomorphism in $\catX$}\}\\
                         &=\{h:X\to Y\text{ in }\catX\mid \text{$Th$  is a left adjoint
                           monomorphism in $\catX^\mT$}\}.
  \end{align*}
\end{remark}

In the sequel we call
\begin{itemize}
\item a morphism $h:X\to Y$ in $\catX$ \df{order-mono} whenever, for all
  $f,g:A\to X$ in $\catX$, $h\cdot f\le h\cdot g$ implies $f\le g$.
\item a morphism $h:X\to Y$ in $\catX$ \df{order-epi} whenever, for all
  $f,g:Y\to B$ in $\catX$, $f\cdot h\le g\cdot h$ implies $f\le g$.
\item a functor $T:\catX\to\catX$ \df{order-faithful} whenever, for all
  $f,g:A\to X$ in $\catX$, $Tf\le Tg$ implies $f\le g$.
\end{itemize}

We denote the class of all order-monos of $\catX$ by $\omono(\catX)$, and the
class of all order-epis by $\oepi(\catX)$. Clearly, if $\catX$ is order-enriched
with the discrete order, then the notions above coincide with mono, epi and
faithful, respectively. Furthermore, order-mono implies mono, order-epi implies
epi and order-faithful implies faithful. The following result is a particular
case of \cite[Proposition~4.1.4]{BF06}.

\begin{proposition}\label{d:prop:1}
  Let $\mT=\monad$ be an order-enriched monad on $\catX$. Then the following
  assertions are equivalent.
  \begin{tfae}
  \item For every object $X$ in $\catX$, $e_X$ is order-mono.
  \item $T$ is order-faithful.
  \end{tfae}
  Moreover, for a Kock-Z\"oberlein monad $\mT$, the two assertions above are
  also equivalent to $\calM_\mT\subseteq\omono(\catX)$.
\end{proposition}

\begin{proof}
  It is immediate, taking into account that, for an order-enriched monad
  $\mT=\monad$, the inequality $Tf\cdot e_X\leq Tg\cdot e_X$ implies $f\leq g$,
  and that the maps $e_X$ belong to $\calM_\mT$.
\end{proof}

\begin{remark}
  For all monads $\mT=\monad$ of Examples~\ref{exs:1}, the functor $T$ is
  order-faithful.
\end{remark}

\section{Abstract algebraic objects}
\label{sec:abstr-algebr-objects}

The role model of this section is the theory of completely distributive and of
totally algebraic lattices in the spirit of \cite{RW94,RW04}. We recall that,
for $\mT=\mD$ being the downset monad on $\POSS$, a partially ordered set $Y$ is
isomorphic to some $DX$ if and only if $Y$ is totally algebraic (see
Definition~\ref{def:domains}~(\ref{item:algebraic})). Analogously, trading the
downset monad for the directed downset monad $\mT=\mI$, a partially ordered set
$Y$ is isomorphic to some $IX$ if and only if $Y$ is an algebraic domain. The
principal observation of this section is that these results are not particularly
about order theory but hold in more general for a Kock-Z\"oberlein monad on an
order-enriched category. To achieve this, an important tool is the equivalence
\cite{RW04} between the category of split algebras and the idempotent split
completion of the Kleisli category which allows us to move back and forth
between these categories.

Hence, in this section we consider a Kock-Z\"oberlein monad $\mT=\monad$ on an
order-enriched category $\catX$. Then the Kleisli category $\catX_\mT$ is
order-enriched as well. Denoting the morphisms of $\catX_\mT$ with arrows
$\modto$ and the composition between them with $\circ$, the canonical functor
from $\catX$ to $\catX_\mT$, given by
\[
  \catX\longto\catX_\mT,\,(f:X\to Y)\longmapsto (f_*=e_Y\cdot f:X\modto Y),
\]
is order-enriched; it is even locally an order embedding provided that $T$ is
order-faithful.  We note that, for arrows $r:A\modto X$ and $s:Y\modto B$ in
$\catX_\mT$ and $f:X\to Y$ in $\catX$,
\begin{align*}
  f_*\circ r=Tf\cdot r &&\text{and}&& s\circ f_*=s\cdot f.
\end{align*}

The following definition is motivated by \cite{Law73}.

\begin{definition}
  An object $Y$ in $\catX$ is called \df{Cauchy complete} whenever every left
  adjoint morphism $r:X\modto Y$ in $\catX_\mT$ is of the form $r=f_*$, for some
  $f:X\to Y$ in $\catX$.
\end{definition}

\begin{remark}
  Equivalently, $Y$ is Cauchy complete if and only if every left adjoint
  $g:TX\to TY$ in $\catX^\mT$ is of the form $g=Tf$, for some $f:X\to Y$ in
  $\catX$.
\end{remark}

\begin{examples}
  For the downset monad $\mD$ on $\POSS$, every partially ordered set $X$ is
  Cauchy complete. For each of the filter monads on $\TOP_0$, a T$_0$-space $X$
  is Cauchy complete if and only if $X$ is sober.
\end{examples}

\begin{theorem}\label{thm:algebra_implies_Cauchy}
  Every $\mT$-algebra $Y$ is Cauchy complete. Moreover, if $T$ is
  order-faithful, an object $Y$ of $\catX$ is a $\mT$-algebra if and only if $Y$
  is Cauchy-complete and $Te_Y$ has a left adjoint in $\catX$.
\end{theorem}
\begin{proof}
  Assume first that $Y$ is a $\mT$-algebra, with left adjoint $\beta:TY\to Y$ of
  $e_Y:Y\to TY$. Since $T$ is order-enriched, also $T\beta\dashv Te_Y$, hence
  $Te_Y$ has a left adjoint. Let $s:Y\modto X$ be the right adjoint of
  $r:X\modto Y$ in $\catX_\mT$. Hence,
  \begin{align*}
    e_X\le s\circ r=m_X\cdot Ts\cdot r &&\text{and}&& e_Y\ge r\circ s=m_Y\cdot Tr\cdot s.
  \end{align*}
  We put $f=\beta\cdot r$, then $f_*=e_Y\cdot\beta\cdot r\ge r$. In fact,
  $f_*\dashv s$ in $\catX_\mT$ since $s\circ f_*\ge s\circ r\ge e_X$ and
  \[
    f_*\circ s=T\beta\cdot Tr\cdot s \le T\beta\cdot e_{TY}\cdot m_Y\cdot
    Tr\cdot s\le T\beta\cdot e_{TY}\cdot e_Y=T\beta\cdot Te_{Y}\cdot e_Y=e_Y;
  \]
  and therefore $r=f_*$.

  Assume now that $T$ is order-faithful and let $Y$ be a Cauchy-complete
  $\catX$-object so that $Te_Y$ has a left adjoint in $\catX$. Then, since
  $Te_Y:TY\to TTY$ corresponds to $(e_Y)_*:Y\modto TY$ in $\catX_\mT$, $(e_Y)_*$
  has a left adjoint $r:TY\modto Y$ in $\catX_\mT$ (see also
  Remark~\ref{d:rem:1}). Since $Y$ is Cauchy complete, $r=\beta_*$ for some
  $\beta:TY\to Y$. Finally, $(-)_*:\catX\to\catX_\mT$ is locally an
  order-faithful by hypothesis, therefore $\beta\dashv e_Y$.
\end{proof}

\begin{corollary}
  Let $\mT=\monad$ be an idempotent monad on a category $\catX$ where $T$ is
  faithful. Then an object $Y$ of $\catX$ is a $\mT$-algebra if and only if $Y$
  is Cauchy complete.
\end{corollary}

We recall now the general notion of a split algebra for a monad as used in
\cite{RW04}.

\begin{definition}
  A $\mT$-algebra $X$ is called \df{split} whenever the left adjoint
  $\alpha:TX\to X$ of $e_X:X\to TX$ has a left adjoint $t:X\to TX$ in $\catX$;
  and $X$ is called \df{algebraic} whenever $X$ is isomorphic to a free algebra
  in $\catX^\mT$.
\end{definition}

\begin{examples}\label{exs:2}
  A partially ordered set $X$ is a split algebra for the downset monad if and
  only if $X$ is completely distributive, in this case the map $t:X\to DX$ sends
  $x\in X$ to the set $\{y\in X\mid y\lll x\}$ of all elements $y\in Y$ which
  are totally below $x$ (see \cite{Ran52,FW90}). Similarly, a directed
  cocomplete partially ordered set $X$ is a split algebra for $\mI=\imonad$ if
  and only if $X$ is a domain (see Definition \ref{def:domains}(4)); in this
  case the splitting $t:X\to IX$ is given by $x\mapsto\{y\in X\mid y\ll
  x\}$. Regarding the filter monad $\mF=\fmonad$ on $\TOP_0$, a continuous
  lattice $X$ (equipped with the Scott topology) is a split algebra for $\mF$ if
  and only if $X$ is $\mF$-disconnected in the sense of \cite{Hof13a}. Here,
  with $\alpha:FX\to X$ denoting the algebra structure of $X$, for an open
  subset $A\subseteq X$ we put
  $\mu(A)=\{x\in X\mid \alpha(\ff)=x\text{ for some $\ff\in FX$ with
    $A\in\ff$}\}$. Then $X$ is $\mF$-disconnected precisely when $\mu(A)$ is
  open, for every open subset $A\subseteq X$; and in this case the map
  $t:X\to FX$ sends $x\in X$ to the filter
  $t(x)=\{A\subseteq X\mid A\text{ open, }x\in\mu(A)\}$. The case of the prime
  filter monad is similar, with $\mu(A)$ now defined using only prime
  filters. In terms of partially ordered compact Hausdorff spaces, every split
  algebra for $\mFp$ is a Priestley space, more precise, a Priestley space is a
  split algebra for $\mFp$ if and only if it is an f-space in the sense of
  \cite{PS88}.  In Section~\ref{sec:idemp-split-compl} we give a different
  characterisation of the split algebras for the filter monad, by means of the
  way below relation. In Examples~\ref{exs:3} we describe algebraic
  $\mT$-algebras.
\end{examples}

We denote the full subcategory of $\catX^\mT$ of all split $\mT$-algebras by
$\Spl(\catX^\mT)$. Since $\mT$ is of Kock-Z\"oberlein type, every free
$\mT$-algebra $TY$ (with algebra structure $m_Y:TTY\to TY$) is split since
$Te_Y\dashv m_Y\dashv e_{TY}$. Hence, every algebraic $\mT$-algebra is
split. Next we recall that the split $\mT$-algebras are precisely those algebras
where the algebra structure has a homomorphic splitting (see \cite{Koc95}).

\begin{proposition}
  Let $X$ be a $\mT$-algebra with $\alpha\dashv e_X$ in $\catX$ and let
  $t:X\to TX$ in $\catX$. Then $t\dashv\alpha$ in $\catX$ if and only if $t$ is
  a $\mT$-homomorphism with $\alpha\cdot t=\id_X$.
\end{proposition}

The following two results exhibit the connection with idempotents in $\catX_\mT$
as shown in~\cite{RW04}.

\begin{proposition}
  For every split $\mT$-algebra $X$ with $t\dashv\alpha\dashv e_X$, $t\le e_X$
  and $t\circ t=t$.
\end{proposition}

Recall that $\catX$ is \emph{idempotent split complete}, or just
\emph{idempotent complete}, whenever every idempotent morphism $e:X\to X$ in
$\catX$ is of the form $s\cdot r$, for some $r:X\to Y$ and $s:Y\to X$ in $\catX$
with $r\cdot s=\id_Y$.  (see \cite{ARV10}, for instance).  Every category with
equalisers or with coequalisers is idempotent split complete.

\begin{theorem}
  Assume that $\catX$ is idempotent split complete. Then $\Spl(\catX^\mT)$ is
  equivalent to the idempotent split completion $\kar(\catX_\mT)$ of
  $\catX_\mT$.
\end{theorem}

In the remainder of this section we aim for a characterisation of algebraic
$\mT$-algebras in an intrinsic way, for idempotent split complete order-enriched
categories $\catX$. Under the equivalence
$\Spl(\catX^\mT)\simeq\kar(\catX_\mT)$, a split algebra $X$ with
$t\dashv\alpha\dashv e_X$ corresponds to $(X,t)$ in $\kar(\catX_\mT)$; in
particular, the free algebra $TY$ corresponds to $(TY,Te_Y)$. Moreover, for
every $Y$ in $\catX$, $(Y,e_Y)\simeq(TY,Te_Y)$ in $\kar(\catX_\mT)$. Hence:

\begin{corollary}
  Assume that $\catX$ is idempotent split complete. A split $\mT$-algebra $X$
  with $t\dashv\alpha\dashv e_X$ is algebraic if and only if
  $(X,t)\simeq (Y,e_Y)$ in $\kar(\catX_\mT)$, for some $Y$ in $\catX$.
\end{corollary}
To describe this condition, we introduce the following notion.

\begin{definition}
  A morphism $f:X\to Y$ in $\catX$ is called \df{$\mT$-dense} whenever
  $f_*:X\modto Y$ has a right adjoint $f^*:Y\modto X$ in $\catX_\mT$.
\end{definition}
\begin{remark}
  Clearly, $f_*:X\modto Y$ has a right ajoint in $\catX_\mT$ if and only if the
  corresponding algebra homomorphism $Tf:TX\to TY$ has a right adjoint in
  $\catX^\mT$. By Remark~\ref{rem:1}, this is equivalent to $Tf$ being left
  adjoint in $\catX$. $\mT$-dense morphisms are studied in 4.3 of \cite{BF06} in
  the realm of completion Kock--Z\"oberlein monads. From
  Proposition~\ref{d:prop:1} we have that, if $T$ is order-faithful,
  \[
    \calM_\mT=\mT\text{-dense}\cap \omono(\catX).
  \]
\end{remark}

\begin{examples}
  \begin{enumerate}
  \item For $\mT=\mD$ being the downset monad on $\POSS$ , every monotone map
    $f:X\to Y$ is $\mD$-dense. In fact, for a monotone map $f:X\to Y$, the right
    adjoint $f^*:Y\modto X$ of $f_*$ in $\POSS_\mD$ is given by
    $f^*(y)=\{x\in X\mid f(x)\le y\}$, for all $y\in Y$.
  \item If we consider the monad $\mI=\imonad$ instead, then $f_*$ has a right
    adjoint if and only if ``$f^*$ lives in $\POSS_\mI$'', that is, if and only
    if $\{x\in X\mid f(x)\le y\}$ is directed, for all $y\in Y$.
  \item For $\mT=\mF$ being the filter monad on $\TOP_0$, every continuous map
    $f:X\to Y$ is $\mF$-dense. Here, for a continuous map $f:X\to Y$, the right
    adjoint $f^*:Y\modto X$ of $f:X\modto Y$ is given by
    $f^*(y)=\langle\{f^{-1}(B)\mid B\in\Omega(y)\}\rangle\in FX$, for all
    $y\in Y$.
  \item For the proper filter monad $\mFO$ on $\TOP_0$, a continuous map
    $f:X\to Y$ is $\mFO$-dense if and only if the filter
    $\langle\{f^{-1}(B)\mid B\in\Omega(y)\}\rangle$ is proper, for each
    $y\in Y$; and this in turn is equivalent to $f$ being dense in the usual
    topological sense.
  \item Similarly, for the prime filter monad $\mFp$ on $\TOP_0$, a continuous
    map $f:X\to Y$ is $\mFp$-dense if and only if the filter
    $\langle\{f^{-1}(B)\mid B\in\Omega(Y)\}\rangle$ is prime. By \cite[Lemma
    6.5]{EF99}, this condition is equivalent to $f$ being flat. More generally,
    for the $n$-prime filter monads $\mF_n$, to be $\mF_n$-dense is equivalent
    to be $n$-flat \cite{CS15}.
  \end{enumerate}
\end{examples}

\begin{assumption}
  From now on we also assume that
  \begin{itemize}
  \item \emph{$\catX$ has equalisers} and
  \item \emph{$T$ sends regular monomorphisms to monomorphisms}.
  \end{itemize}
\end{assumption}

Since $\catX$ has equalisers, $\catX$ is also idempotent split complete. We
remark that these conditions are satisfied in all Examples~\ref{exs:1}.

\begin{lemma}\label{lem:i}
  If $i:A\to X$ is a regular monomorphism in $\catX$, then $i_*$ is a
  monomorphism in $\catX_\mT$.
\end{lemma}
\begin{proof}
  Just observe that $i_*\circ r=i_*\circ s$ in $\catX_\mT$ translates to
  $Ti\cdot r=Ti\cdot s$ in $\catX$.
\end{proof}

\begin{proposition}\label{prop:char_algebraic}
  Let $X$ be a split $\mT$-algebra with $t\dashv\alpha\dashv e_X$ and let
  \[
    \xymatrix{A\ar[r]^i & X\ar@<.5ex>[r]^{e_X}\ar@<-.5ex>[r]_t & TX}
  \]
  be an equaliser diagram. Then the following assertions hold.
  \begin{enumerate}
  \item $i_*:A\modto X$ is a morphism of type $i_*:(A,e_A)\modto(X,t)$ in
    $\kar(\catX_\mT)$.
  \item $X$ is algebraic if and only if $i:A\to X$ is $\mT$-dense and
    $i_*\circ i^*=t$.
  \end{enumerate}
\end{proposition}
\begin{proof}
  To show the first assertion, we calculate
  \[
    t\circ i_*=t\cdot i=e_X\cdot i=Ti\cdot e_A=i_*\circ e_A.
  \]
  Regarding the second assertion, assume first that $X$ is algebraic, that is,
  there are arrows $r:(Y,e_Y)\modto(X,t)$ and $s:(X,t)\modto(Y,e_Y)$ in
  $\kar(\catX_\mT)$ with $s\circ r=e_Y$ and $r\circ s=t$. Since $t\le e_X$, we
  conclude that $r\dashv s$ in $\catX_\mT$ and, since the $\mT$-algebra $X$ is
  Cauchy complete (see Theorem~\ref{thm:algebra_implies_Cauchy}), $r=f_*$ for
  $f=\alpha\cdot r:Y\to X$. Furthermore,
  \[
    t\cdot f=t\cdot\alpha\cdot r=m_X\cdot Tt\cdot r=t\circ r=r=f_*=e_X\cdot f,
  \]
  hence there is an arrow $h:Y\to A$ in $\catX$ with $i\cdot h=f$. Then
  \[
    i_*\circ h_*\circ s=f_*\circ s=r\circ s=t\le e_Y
  \]
  and
  \[
    i_*\circ h_*\circ s\circ i_*=t\circ i_*=i_*\circ e_A,
  \]
  hence $ h_*\circ s\circ i_*= e_A$, by Lemma~\ref{lem:i}. Putting
  $i^*= h_*\circ s$, we have seen that $i_*\dashv i^*$ in $\catX_\mT$ and
  $i_*\circ i^*=t$.

  Conversely, assume now that $i_*$ has a right adjoint $i^*$ with
  $i_*\circ i^*=t$. Since
  \[
    i^*\circ t=i^*\circ i_*\circ i^*=i^*=e_A\circ i^*,
  \]
  $i^*:(X,t)\modto(A,e_Y)$ is a morphism in $\kar(\catX_\mT)$; it is indeed an
  isomorphism since $i^*\circ i_*=e_A$ and $i_*\circ i^*=t$.
\end{proof}

Finally, we can simplify the condition $i_*\circ i^*=t$ and obtain:

\begin{theorem}\label{thm:char_algebraic}
  With the same assumption as in Proposition~\ref{prop:char_algebraic}, $X$ is
  algebraic if and only if $i:A\to X$ is $\mT$-dense and $\alpha\cdot Ti$ is an
  epimorphism in $\catX$.
\end{theorem}
\begin{proof}
  For $r:X\modto Y$ in $\catX_\mT$, we write $\widehat{r}:TX\to TY$ for the
  corresponding $\mT$-algebra homomorphism. With the notation of
  Proposition~\ref{prop:char_algebraic}, if $i:A\to X$ is $\mT$-dense with right
  adjoint $i^*$, then $i_*\circ i^*=t$ if and only if
  $\widehat{t}=\widehat{i_*}\cdot\widehat{i^*}=Ti\cdot\widehat{i^*}$ if and only
  if the $\mT$-algebra homomorphisms $Ti:TA\to TX$ and $\widehat{i^*}:TX\to TA$
  split the idempotent $\widehat{t}:TX\to TX$. But since $\widehat{t}:TX\to TX$
  is also split by $\alpha:TX\to X$ and $t:X\to TX$, $i_*\circ i^*=t$ if and
  only if
  \begin{align*}
    \widehat{i^*}\cdot t\cdot\alpha\cdot Ti=\id_{TA} && \text{and} && \alpha\cdot Ti\cdot \widehat{i^*}\cdot t=\id_X.
  \end{align*}
  Furthermore, the first equality is always true:
  \[
    \widehat{i^*}\cdot t\cdot\alpha\cdot Ti=\widehat{i^*}\cdot m_X\cdot Tt\cdot
    Ti=\widehat{i^*}\cdot m_X\cdot Te_X\cdot Ti=\widehat{i^*}\cdot Ti=\id_{TA};
  \]
  therefore the second one holds precisely when $\alpha\cdot Ti$ is an
  epimorphism in $\catX$.
\end{proof}

\begin{examples}\label{exs:3}
  We continue here Examples~\ref{exs:2}.
  \begin{enumerate}
  \item For $\mT=\mD$ being the downset monad on $\POSS$,
    Theorem~\ref{thm:char_algebraic} tells us that a completely distributive
    lattice $L$ is algebraic for $\mD$ if and only if $L$ is totally algebraic,
    that is, if every element is the supremum of all the elements totally below
    it.
  \item We consider now the directed downset monad $\mI=\imonad$ on $\POSS$. In
    this case, a directed cocomplete partially ordered set $X$ is a split
    algebra if and only if it is a domain; in this case the splitting
    $t:X\to IX$ is given by $x\mapsto\{y\in X\mid y\ll x\}$. Moreover, $X$ is
    algebraic if and only if, for every $x\in X$, the set
    $\{y\in X\mid y\ll y\ll x\}$ is directed and has $x$ as supremum; that is,
    if $X$ is algebraic in the sense of domain theory (see \cite{AJ94}).
  \item Let now $X$ be a $\mF$-disconnected continuous lattice. Then the
    elements of $A$ are precisely those elements $x\in X$ where, for all open
    subsets $B\subseteq X$, $x\in\mu(B)$ implies that $x\in B$. Then $X$ is
    algebraic if and only if every $x\in X$ is the largest convergence point
    (with respect to the specialisation order) of a filter $\ff\in Fi[FA]$.
  \item Similarly, an f-space $X$ is a algebraic for the prime filter monad
    $\mFp$ if and only if every $x\in X$ is the largest convergence point (with
    respect to the specialisation order) of a prime filter $\ff\in F_2i[F_2 A]$.
  \end{enumerate}
\end{examples}

\section{Weighted (co)limits and cogenerators}
\label{sec:weight-colim-cogen}

``Cocompleteness almost implies completenes'' is the title of the paper
\cite{AHR89} of Ji\v r\'{\i} Ad\'amek, Horst Herrlich and Ji\v r\'{\i}
Reiterman, as well as the main theme of section 12 of the book~\cite{AHS90}. The
title announces several results giving conditions under which completeness and
cocompleteness are equivalent.  In particular, it is proved (the dual of) that a
complete and wellpowered category with a cogenerator is cocomplete (and
co-wellpowered).

In the setting of order-enriched categories, it is natural to consider
``order-enriched'' limits and colimits, the so-called \df{weighted (co)limits},
or \df{indexed (co)limits}. Thus, the question of knowing when weighted
completeness does imply weighted cocompleteness arises. Here we show that it
happens in the presence of a regular cogenerator.

\begin{remark}\label{r-weighted}
  \begin{enumerate}
  \item We start by recalling the notion of weighted limit (\cite{Kel82}) in the
  order-enriched setting. Let $D:\catD\ra \catX$ and $W:\catD\ra \POSS$ be
  order-enriched functors, with $\catD$ small.  They give rise to the functor
  $\POSS^{\catD}(W,\, \catX(-, D))$ from $\catX^{\op}$ to $\POSS$, where, for
  every $X\in \catX$, $\catX(-,D)(X)$ stands for the functor
  $\catX(X,-)\cdot D:\catD\ra \POSS$.  The \df{limit of $D$ weighted by $W$}, in
  case it exists, is an object $L$ of $\catX$ which represents that functor,
  that is, there is a natural isomorphism
  \begin{equation}\label{iso}
    \catX(-,L) \cong \POSS^{\catD}(W,\,\catX(-,D)).
  \end{equation}
  This is equivalent to say that we have a family of morphisms
  \[
    L\xrightarrow{\;l_d^x\;}Dd, \; d\in \catD,\, x\in Wd
  \]
  satisfying the following conditions:
  \begin{enumerate}
  \item $l_d^x\leq l_d^y$ whenever $x\leq y$, and
    $Dn\cdot l_d^x=l_{d'}^{Wn(x)}$, for all morphisms $n:d\ra d'$ in $\catD$ and
    all $x\in Wd$. (This gives the natural transformation from $W$ to
    $\catX(L, D)$ which is the image of $\id_L$ by the component of the natural
    transformation indexed by $L$.)
  \item The family $\left(l_d^x\right)_{d\in \catD, x\in Wd}$ is universal,
    i.e., the natural transformation \eqref{iso} is a natural isomorphism. This
    means that every family of morphisms
    $A\xrightarrow{a_d^x}Dd, \; {d\in \catD}$, $x\in Wd$, satisfying (a) --
    with $A$ and $a$ in the place of $L$ and $l$ -- is of the form
    $a_d^x=l_d^xt$ for a unique $t:A\ra L$; and, moreover, for $t,t':A\ra L$,
    the inequality $l_d^xt\leq l_d^x t'$, for all $d$ and $x$, imply $t\leq t'$.
  \end{enumerate}
  When $W$ is just the constant functor into a singleton, we speak of
  \df{conical limits}. Thus, a conical limit is a limit in the ordinary sense
  whose projections are jointly order-monic.

  Inserters and cotensor products are special types of weighted limts.  The
  \df{inserter} of a pair of morphisms $f,g:X\ra Y$ is just a morphism
  $i:I\ra X$ with $fi\leq gi$ and universal with respect to that property (in
  the sense of (b) above). Given a poset $I$ and an object $X$ of $\catX$, the
  \df{cotensor product} of $I$ and $X$, denoted by $\cote(I,X)$, is a weighted
  limit with the domain $\catD$ of the functors $D$ and $W$ the unit category,
  i.e., the category with just an object and the corresponding identity
  morphism. Thus, the projections of the cotensor product are of the form
  \[
    \xymatrix{\cote(I,X)\ar[rr]^{l^i}&&X},\; i\in I,
  \]
  with $l_i\leq l_{j}$ for $i\leq j$.

  In an order-enriched category, the existence of conical products and inserters
  guarantees the existence of all weighted limits.

  The dual notions for weighted limits, inserters and cotensor products are,
  respectively, weighted colimits, coinserters and tensor products.

  \item For every Kock-Z\"{o}berlein monad $\mT$ over a category $\catX$ with
  weighted limits, the subcategory $\xupt$ is closed under them (since the
  forgetful functor from $\xupt$ to $\catX$ creates weighted limits). Indeed, as
  shown in \cite{ASV15}, more than being closed under weighted limits, the
  subcategory $\xupt$ is also an inserter-ideal. This means that, for every
  diagram
  \[
    \xymatrix{I\ar[r]^i &A\ar@<1ex>[r]^g\ar@<-1ex>[r]_f&B}
  \]
  with $i$ the inserter of the pair $(f,g)$ in $\catX$, if $f$ is a morphism of
  $\xupt$, then $i:I\ra A$ lies in $\xupt$ too.
  \end{enumerate}
\end{remark}

\begin{remark}\label{r-cogen}
  We also use the ``order-enriched'' version of the notion of cogenerator. In
  this paper, an object $S$ of an order-enriched category is said to be a
  \df{cogenerator} if it detects the order, in the sense that, for every pair of
  morphisms $f,g:X\ra Y$, $f\leq g$ iff $tf\leq tg$ for all morphisms
  $t:Y\ra S$. It follows easily from the definition that a strong cogenerator in
  the sense of 3.6 of \cite{Kel82} is a cogenerator in our sense, provided that
  the category has coinserters. Next we give the notion of \emph{regular
    cogenerator}. (Co)generators in this sense were considered for instance in
  \cite{KV16a}.
\end{remark}

\begin{remark}\label{r-reg_cogen}
\begin{enumerate}
  \item We recall that an order-enriched adjunction between order-enriched
  categories is an adjunction $F\dashv U:\catA\ra \catB$ with $U$ and $F$
  order-enriched, and for which there exists a natural isomorphism between the
  functors $\catB(-,U-)$ and $\catA(F-,-)$ from $\catB\times \catA$ to
  $\POSS$. This is equivalent to say that we have an adjunction
  $F\dashv U:\catA\ra \catB$ with $U$ order-enriched, and the unit $\eta$
  satisfies the property that any inequality of the form
  $Uf\cdot \eta_X\leq Ug\cdot \eta X$, for $f$ e $g$ with common domain and
  codomain, implies $f\leq g$ (\cite{CS15}). Clearly, an order-enriched
  adjunction induces an order-enriched monad; and, for an order-enriched monad
  $\mT$, the adjunctions $F^{\mT}\dashv U^{\mT}$ and $F_{\mT}\dashv U_{\mT}$ are
  order-enriched.

  \vskip1mm

  \item  In an order-enriched category $\catA$ with weighted limits, given an
  object $S$, the cotensor product yields a functor
  \begin{equation}\label{cote}
    \cote(-, S):\POSS\To \catA^{\op}
  \end{equation}
  which is an order-enriched left adjoint of $\catA(-,S)$. For every
  $X\in \catA$, the counit map is given by (the dual of) the morphism $n_X$
  determined by the universality of the cotensor product:
  \begin{equation}\label{eq4.1}
    \xymatrix{
      X\ar[rr]^{n_X\qquad \quad  }\ar[ddr]_{f} & & \cote(\catA(X,S),\s)\ar[ddl]^{\pi_f} & \\
      & & & \hspace{-1cm} f\in \catA(X,S)\\
      & \s & &
    }
  \end{equation}
  Given $X\in \catA$, put
  \[
    \hat{X} = \cote(\catA(X,S),S)
  \]
  and consider the cotensor product
  \[
    \cote(\catA(\hat{X},S),S)
    \xrightarrow{\hspace{0.4cm}\hat{\pi}_g\hspace{0.4cm}} \s, \hspace{0.1cm}
    \;\; g\in \catA(\hat{X},S).
  \]
  Let $\beta: \cote(\catA(X,\s),S) \To \cote(\catA(\hat{X},S),S)$ be
  the unique morphism of $\catA$ which makes the following diagrams commutative:
  \[
    \xymatrix{
      \hat{X}\ar[rr]^{\beta\qquad \quad }\ar[ddr]_{\pi_{g\cdot n_X}} & & \cote(\catA(\hat{X},S),S)\ar[ddl]^{\hat{\pi}_g} &\\
      & & & \hspace{-0.5cm} g\in \catA(\hat{X},S).\\
      & \s & & }
  \]
  Thus, putting $\alpha=n_{\hat{X}}$, we have the diagram
  \begin{equation}\label{eq4.2.0}
    \xymatrix{
      X\ar[r]^{n_X}
      & \hat{X}\ar@<3pt>[r]^-{\beta}\ar@<-3pt>[r]_-{\alpha}
      & \cote(\catA(\hat{X},S),S).
    }
  \end{equation}
\end{enumerate}
\end{remark}

\begin{definition}
  Let $\catA$ be an order-enriched category with weighted limits. An object $S$
  of $\catA$ is said to be a \df{regular cogenerator} if the the diagram
  \eqref{eq4.2.0} is an equaliser.
\end{definition}
If $\catA$ has weighted limits, every equaliser of $\catA$ is conical; hence, it
is immediate that every regular cogenerator detects the order, so, in
particular, it is a cogenerator. In other words, it detects not only equality
between pairs of morphisms, as in the ordinary case, but also inequality.

\begin{theorem}\label{com->cocom}
  Every order-enriched category with weighted limits and a regular cogenerator
  has weighted colimits.
\end{theorem}

\begin{proof}
  Let $\catA$ be an order-enriched category with weighted limits and a regular
  cogenerator $S$. Then, as seen in Remark~\ref{r-reg_cogen}, the functors
  \[
    \xymatrix{ \catA^{\op}\ar@<2pt>[rr]^{\catA(-,S)}& &
      \POSS\ar@<2pt>[ll]^{\cote(-,S)} }
  \]
  form an order-enriched adjunction.  Let $\mT$ be the corresponding monad and
  let $K: \catA^{\op}\to \POSS^{\mT}$ be the comparison functor:
  \[
    \xymatrix{
      \catA^{\op}\ar[rr]^{K}\ar[dr]_{\catA(-,S)} & & \POSS^{\mT}\ar[dl]\\
      & \POSS & }
  \]
  Since $S$ is a regular cogenerator, the morphism $n_{X}^{\op}$, which is,
  pointwisely, the counit of the adjunction $\cote(-,S)\dashv \catA(-,S)$,
  is a regular epimorphism, and, consequently, $K$ is a full and faithful right
  adjoint.  Moreover, this adjunction is order-enriched, as it is explained in
  the next paragraph.
	
  Let $F\dashv U:\catC\to \catB$ be an order-enriched adjunction with the counit
  being pointwisely a conical coequaliser, and $\catC$ having conical
  coequalisers.  It is well-known that, under these conditions, the comparison
  functor $K$ is a full and faithful right adjoint \cite{MS04}. It is clear that
  $K$ is order-enriched. Then, in order to conclude that the adjunction
  $K:\catC\to \catB^{\mT}$ is order-enriched, it suffices to show that, for
  every universal map $\eta_{(X,\xi)}^{\mT}:(X,\xi)\to KA$ of the adjunction,
  and every pair $f,g:A\to B$ of morphisms in $\catC$ with
  $Kf\cdot\eta_{(X,\xi)}^{\mT} \leq Kg.\eta_{(X,\xi)}^{\mT}$, we have $f\leq g$
  (see Remark \ref{r-reg_cogen}.1). Recall that, given $(X,\xi)\in \catB^{\mT}$,
  the universal map from $(X,\xi)$ to $K$ is obtained as follows: take the
  coequaliser $c: FX\to A$ of the pair
  \[
    \xymatrix{ FUFX\ar@<3pt>[r]^{\eps_{FX}}\ar@<-3pt>[r]_{F\xi} & FX }
  \]
  where $\eps$ is the counit of the adjunction $F\dashv U$.  Then
  $Uc\cdot UF\xi = Uc\cdot U\eps_{FX}$. But
  $\xi = \text{coeq}(U\eps_{FX},UF\xi)$; hence, there is a unique
  $\theta: X\to UA$ making the following triangle commutative:
  \[
    \xymatrix{
      UFX\ar[r]^{\xi}\ar[dr]_{Uc} & X\ar[d]^{\theta} \\
      & UA }
  \]
  and it holds $\theta = Uc.\eta_{X}$. It is known that
  \[
    \eta_{(X,\xi)}^{\mT} = \theta
  \]
  and, for every $g: (X,\xi)\to KB$ in $\catB^{\mT}$, the unique
  $\bar{g}:A\to B$ in $\catC$ making the triangle
  \[
    \xymatrix{
      (X,\xi)\ar[r]^{\theta}\ar[d]_{g} & KA\ar[dl]^{K\bar{g}} \\
      KB }
  \]
  commutative is characterised by the equality
  \[
    \bar{g}\cdot c = \eps_{B}\cdot Fg.
  \]
  We show that, given two morphisms $g,h:(X,\xi)\to KB$ with $g\leq h$ then
  $\bar{g}\leq \bar{h}$.  Since $F$ is order-enriched, the inequality $g\leq h$
  implies $\eps_{B}\cdot Fg \leq \eps_{B}\cdot Fh$. But then
  \[
    \bar{g}\cdot c = \eps_{B}\cdot Fg \leq \eps_{B}\cdot Fh = \bar{h}\cdot c
  \]
  and, since $c$ is order-epic (because it is a conical coequaliser),
  $\bar{g}\leq \bar{h}$.

  \vskip0.5mm
	
  Now we have that, for our comparison functor $K: \catA^{\op}\to \POSS^{\mT}$,
	
  \begin{itemize}
  \item $K$ is the right adjoint of an order-enriched adjunction;
  \item $K$ is full and faithful, and it is full with respect to the order, that
    is, given a pair of morphisms
    $\xymatrix{\nonumber\ar@<-3pt>[r]_{f}\ar@<3pt>[r]^{g} & }$ in $\catA^{\op}$,
    $f\leq g$ in $\catA^{\op}$ iff $Kf\leq Kg$ in $\POSS^{\mT}$.
  \end{itemize}
	
  Consequently, since $\POSS^{\mT}$ has weighted limits, also $\catA^{\op}$ has
  weighted limits, and the weighted limits in $\catA^{\op}$ are constructed, up
  to isomorphism, as in $\POSS^{\mT}$. (This can be easily proved in a way
  analogous to the one of the ordinary case.)  That is, $\catA$ has weighted
  colimits.
\end{proof}

In the next section we apply this theorem to the categories $\ALat$ of algebraic
lattices with maps which preserve directed suprema and all infima, the category
$\ADom$ of bounded complete algebraic domains with maps which preserve directed
suprema and all non-empty infima, and the category $\SPEC$ of spectral
topological spaces and spectral maps.

\section{(Co)completeness of subcategories of $\xupt$}
\label{sec:cocompl-subc-xupt}

In this and the next section we work within the category $\TOP_0$ of $T_0$
topological spaces and continuous maps. We consider the relation $\leq$ in a
space to be the specialisation order, and we use also the symbol $\leq$ to refer
to the corresponding order induced in the hom-sets of $\TOP_0$. We do this in
order to fit our terminology on continuous domains and lattices with
\cite{GHK+03}. Thus, as mentioned already in Examples \ref{exs:1}, the open
filter monads are KZ with respect to $\geq$.

The category $\TOP_0$ has weighted limits, since its ordinary limits are
conical, and the inserter of a pair $(f,g)$ of morphisms with domain in $X$ is
just the subspace of all $x\in X$ with $f(x)\leq g(x)$. Therefore, for every
{\KZ} monad $\mT$ over $\catX=\TOP_0$ the corresponding Eilenberg-Moore category
$\xupt$ is closed under weighted limits in $\TOP_0$ (since the forgetful functor
from $\xupt$ to $\TOP_0$ creates limits).  Hence, the cotensor product yields
the functor
\[
  \cote(-, \s):\POSS\To \left(\xupt\right)^{\op}.
\]
This functor is defined as in \eqref{cote} of the previous section with $\s$
denoting the Sierpi\'nski space. And we can consider the diagram defined as in
\eqref{eq4.2.0}:
\begin{equation}\label{eq4.2}
  \xymatrix{
    X\ar[r]^{n_X} & \hat{X}\ar@<3pt>[r]^{\hspace{-1.2cm}\beta}\ar@<-3pt>[r]_{\hspace{-1.2cm}\alpha} & \cote(\Hhom(\hat{X},\s),\s)
  }
\end{equation}
where $\Hhom$ refers to hom-posets of $\xupt$.  Let
\[
  \catX_{\alg}
\]
denote the full subcategory of $\xupt$ for which the diagram \eqref{eq4.2} is an
equaliser in $\TOP_0$, then also in $\xupt$.

We are going to show that, concerning the filter, the proper filter and the
prime filter monads, the subcategories $\catX_{\alg}$ are well-known categories,
namely: the category $\ALat$ of algebraic lattices with maps which preserve
directed suprema and all infima, the category $\ADom$ of bounded complete
algebraic domains with maps which preserve directed suprema and all non-empty
infima, and the category $\SPEC$ of spectral topological spaces and spectral
maps (see Definition \ref{def:domains}).  We show that all of them are closed
under weighted limits. Hence, the equaliser diagram \eqref{eq4.2} tells us that
the Sierpi\'nski space is a regular cogenerator of $\catX_{\alg}$. Moreover, it
allows us to conclude that:
\begin{enumerate}
\item[(1)] $\catX_{\alg}$ is the closure under weighted limits of $\xdot$ in
  $\xupt$, and in $\TOP_0$ (Corollary \ref{cor4.4});
\item[(2)] $\catX_{\alg}$ has weighted colimits (Corollary \ref{cor4.6}).
\end{enumerate}

We start by establishing the closedness under weighted limits:

\begin{proposition}\label{prop4.1}
  Every one of the three categories, $\ALat$, $\ADom$ and $\SPEC$, is closed
  under weighted limits in $\TOP_0$.
\end{proposition}

\begin{proof}
  Let $\mT$ be a {\KZ} monad over $\TOP_0$; then $\xupt$ is closed under
  weighted limits in $\TOP_0$.  Inserters in $\TOP_0$ are topological
  embeddings, then also order embeddings. Thus, the same happens in $\xupt$.
	
  Let now $\mathcal{K}$ be a full subcategory of $\xupt$. Then, in order to
  ensure that $\mathcal{K}$ is closed under weighted limits in $\TOP_0$, it
  suffices to show that $\mathcal{K}$ is closed in $\xupt$ under
	
  \begin{itemize}
  \item (conical) products, and
  \item topological embedding subobjects, i.e., for every topological embedding
    \newline $m:X\hookrightarrow Y$ in $\xupt$ with $Y$ in $\mathcal{K}$, also
    $X$ belongs to $\mathcal{K}$.
  \end{itemize}
	
  \noindent Since for the filter and the proper filter monads the morphisms of
  $\xupt$ are the maps preserving directed suprema and infima (respectively,
  non-empty infima), the closedness under products and topological embedding
  subobjects of $\ALat$ and $\ADom$ in the corresponding category $\xupt$ is
  just Proposition I-4.12 and Corollary I-4.14 of \cite{GHK+03}.

  Concerning the category $\SPEC$, we observed already in
  Definition~\ref{def:domains}~(\ref{item:stablycompact}) and
  Example~\ref{exs:1}(\ref{item:StablyComp}) that $\SPEC$ is a reflecive full
  subcategory of $\xupt\simeq\STCOMP$ since it is closed in it under initial
  cones. In particular, it is closed under products and embeddings.
\end{proof}

Next we show that the Sierpi\'nski space $\s$ is a regular cogenerator for each
one of the three categories, $\ALat$, $\ADom$ and $\SPEC$. For that, we first
prove Lemma \ref{lem4.2} below, where we present a common feature of the three
categories, which gives the means for the proof of Theorem \ref{theo4.3}.

Before stating that lemma, we describe the morphism $n_X:X\To \hat{X}$, defined
in \eqref{eq4.1}, in any full subcategory $\catA$ of $\xupt$ closed under
weighted limits and containing the Sierpinski space $\s$.  Given $X\in \catA$,
let
\[
  \Lambda X = \{U\in \Omega X \mid \mathcal{X}_U : X \To \s \text{ is a morphism
    of } \catA\}.
\]
Then $\hat{X}=\cote(\Hhom(X,\s),\s)$ consists of all families
$(z_U)_{U\in \Lambda X}$ in the product $\s^{\Lambda X}$ with the property
$U\subseteq V\Rightarrow z_U\leq z_V$, and
$n_X(x)=(\chi_U(x))_{U\in \Lambda X}$.  The topology of $\hat{X}$ is just the
one induced by the product topology. Thus, it is generated by the sub-base of
all sets
\[
  \Diamond U=\pi^{-1}_{\chi_U}(\{1\})=\{(z_U)_{U\in \Lambda X}\mid z_U=1\},\;
  U\in \Lambda X,
\]
and we have $U=n_X^{-1}(\Diamond U)$. Moreover, since the projections
$\pi_{\chi_{_U}}$ belong to $\Hhom(X,\s)$, the sets $\Diamond U$ belong to
$\Lambda \hat{X}$.

\begin{lemma}\label{lem4.2}
  Let $\catA$ be one of the categories $\ALat$, $\ADom$ or $\SPEC$. Then $\catA$
  satisfies the following conditions:
  \begin{tfae}
  \item The spaces of $\catA$ are sober and $\s \in \catA$.
  \item $\catA$ is closed under weighted limits in $\TOP_0$.
  \item For every $X \in \catA$, the set $\Lambda X$ is closed under
    finite intersections (in particular, contains $X$) and forms a base of the
    topology $\Omega X$.
  \item For every $X\in \catA$, the morphism $n_X : X\to \hat{X}$ has the
    following property, for every family $V_i$, $i\in I$, of sets of
    $\Lambda X$:
    \[
      \text{If } H=\bigcup_{i\in I}V_i \in \Lambda X, \text{ then } H =
      n^{-1}_X(H'), \text{ for some $H' \in \Lambda \hat{X}$ with
        $H'\subseteq \bigcup_{i\in I}\Diamond V_i$}.
    \]
  \end{tfae}
\end{lemma}

\begin{proof}
  Condition (i) is well-known for the three categories.
	
  Condition (ii) is Proposition~\ref{prop4.1}.
	
  We show condition (iii) for $\ALat$. Given $X\in \ALat$ and $U\in \Omega X$,
  the characteristic function $\chi_U : X\to \s$ is a morphism of $\ALat$ iff it
  preserves arbitrary infima, and this is equivalent to $U$ being closed under
  arbitrary infima. We show that it forms a base of $\Omega X$.  If $U$ is
  closed under infima, it is of the form $U = \upc c$ where
  $c= \bigwedge U$. But then the open sets of $X$ closed under infima are
  precisely all of the form $\upc c$ with $c$ a compact element of $X$, and
  these sets are known to be a base for the topology of the algebraic lattice
  $X$. Moreover, they are closed under finite intersections.
	
  Condition (iii) for $\ADom$ is shown in an analogous way and we have, in this
  case,
  \[
    \Lambda X = \{U\in \Omega X\mid U \text{ is closed under non-empty
      infima}\}.
  \]

  Concerning (iii) for $\SPEC$, it is obvious that a continuous map $f: X\to \s$
  is spectral iff $f^{-1}(\{1\})$ is compact.  Thus
  \[
    \Lambda X = \{U\in \Omega X\mid U \text{ is compact}\}
  \]
  which is, by definition of spectral space, a base of $\Omega X$.
	
  Now we verify condition (iv) for the three categories.

  \vskip1.5mm
	
  $\catA = \ALat$. Let $H = \bigcup_{i\in I} V_i$ belong to $\Lambda X$ with all
  $V_i$ in $\Lambda X$. Then, $\bigcup_{i\in I} V_i = \upc~a$, with $a$ a
  compact element of $X$; hence, $a\in V_{i_0}$ for some $i_0 \in I$; but
  $V_{i_0} = \upc V_{i_0}$, thus we have $\bigcup_{i\in I} V_i =
  V_{i_0}$. Consequently,
  \[
    H=V_{i_0} = n^{-1}(\Diamond V_{i_0}) \text{ with } \Diamond V_{i_0}
    \subseteq \bigcup_{i\in I} \Diamond V_i.
  \]

  \vskip1.5mm

  $\catA=\ADom$.  The same proof as for $\ALat$, in case
  $\bigcup_{i\in I} V_i \neq \varnothing$. The case
  $\bigcup_{i\in I} V_i = \varnothing$ is trivial.

  \vskip2.5mm

  $\catA=\SPEC$. Consider $H=\bigcup_{i\in I} V_i$ in $\Lambda X$, with
  $V_i\in \Lambda X,\;\;i\in I$. Then, since $\bigcup_{i\in I} V_i$ is compact,
  it can be written as $\bigcup_{i\in I} V_i = \bigcup_{j\in J} V_j$, with
  $J\subseteq I$ finite.  Hence, we obtain
  \[
    H = \bigcup_{i\in I} V_i = \bigcup_{j\in J} V_j = n^{-1}(\bigcup_{j\in J}
    \Diamond V_j)
  \]
  wth $\bigcup_{j\in J} \Diamond V_j \subseteq \bigcup_{i\in I} \Diamond V_i$,
  and $\bigcup_{j\in J}\Diamond V_j \in \Lambda \hat{X}$, because it is a finite
  union of compact open sets of $\hat{X}$.
\end{proof}

\begin{theorem}\label{theo4.3}
  For a subcategory $\catA$ of $\TOP_0$ fulfilling conditions (i)-(iv) of Lemma
  \ref{lem4.2}, the diagram \eqref{eq4.2} is an equaliser in $\TOP_0$. As a
  consequence, the Sierpi\'nski space is a regular cogenerator in $\catA$, and,
  in particular, in each one of the categories $\ALat$, $\ADom$ and $\SPEC$.
\end{theorem}

\begin{proof}
  We prove that if $\catA$ is a subcategory of $\TOP_0$ fulfilling conditions
  (i)-(iv) of Lemma~\ref{lem4.2}, then \eqref{eq4.2} is an equaliser in
  $\TOP_0$. Since $\catA$ contains $\s$ and is closed under weighted limits in
  $\TOP_0$, it immediatly follows that \eqref{eq4.2} is also an equaliser in
  $\catA$.
	
  Put $n=n_X$. In order to conclude that $n$ is indeed the equaliser of $\alpha$
  and $\beta$, let
  \[
    Y \xrightarrow{\hspace{0.3cm}h\hspace{0.3cm}} \hat{X}=\cote
    (\Hhom(X,\s), \s)
  \]
  be a morphism in $\TOP_0$ such that
  \[
    \alpha h = \beta h.
  \]
  For $y\in Y$, put
  \[
    h(y)=(y_{U})_{U\in \Lambda X}.
  \]
  
  We show that:
	
  \begin{enumerate}
  \item[(A)] For every $y\in Y$, the set
    \[
      F_y = \{U\in \Lambda X\mid y_{_U} = 1\}
    \]
    is a filter of the poset $(\Lambda X, \subseteq)$, and has the following
    property:
    \begin{equation}
    \text{If $\bigcup_{i\in I} V_i\in F_y$ with all
      $V_i\in \Lambda X$, then $V_j\in F_y$ \text{ for some } $j\in I$.}
    \tag{$\Diamond$}
    \end{equation}

  \item[(B)] Every filter $F$ of the poset $(\Lambda X, \subseteq)$ satisfying
    property ($\Diamond$) is of the form
    \[
      F = B(x) = \{U\in \Lambda X\mid x\in U\}
    \]
    for a unique $x\in X$.

  \end{enumerate}
  \noindent
  After proving (A) and (B), it is then clear that we can define
  $\bar{h}: Y\to X$ by putting
  \[
    \bar{h}(y) = x\;\;\;\;\;\text{with}\;\; F_y = B(x),
  \]
  and this is the unique map making the triangle
  \[
    \xymatrix{
      X\ar[rr]^{n} & & \hat{X} \\
      & Y\ar@{-->}[lu]^{\bar{h}}\ar[ru]_{h} & }
  \]
  commutative. The fact that $\bar{h}$ is continuous follows, since $n$ is a
  topological embedding.
	
  \noindent \textit{Proof of (A)}. We observe that the equality
  $\alpha h(y) = \beta h(y)$ means that
  \[
    \chi_{H}({(y_{_U})}_{U\in \Lambda X}) = y_{n^{-1}(H)}\;\, ,\; \; H\in
    \Lambda \hat{X}.
  \]
  	
  Thus $F_y \neq \varnothing$, because
  $y_{_X} = y_{n^{-1}(\hat{X})} = \chi_{\hat{X}}({(y_{_U})}_{U\in \Lambda X}) =
  1$.
	
  It is also clear that if $U$ and $V$ are two open sets of $\Lambda X$ with
  $U\subseteq V$ and $U\in F_y$ then $V\in F_y$, by definition of
  $\hat{X}$. Moreover, $F_y$ is closed under binary intersections: $V$ and $W$
  laying in $F_y$ means that $y_{_V} = 1$ and $y_{_W} = 1$, that is,
  $(y_{_U})_{U\in \Lambda X}\in (\Diamond V) \bigcap (\Diamond W)$. But then
  $\chi_{(\Diamond V)\bigcap (\Diamond W)} ((y_{_U})_{U\in \Lambda X}) =
  1$. Now, $(\Diamond V) \bigcap (\Diamond W)\in \Lambda \hat{X}$, because
  $\hat{X}\in \catA$ (since $\s\in \catA$ and $\catA$ is closed under weighted
  limits), thus $\hat{X}$ satisfies (iii). Then, we have
  $y_{_{V\bigcap W}} = y_{n^{-1}(\Diamond V \bigcap \Diamond W)} = 1$, that is,
  $V\bigcap W\in F_y$.
	
  We show now that $F_y$ satisfies ($\Diamond$). Let $V_i,\; i\in I$, be a
  family of sets of $\Lambda X$ with $\bigcup_{i\in I} V_i\in F_y$, that is,
  $\bigcup_{i\in I} V_i\in \Lambda X$ and $y_{_{\bigcup_{i\in I} V_i}} = 1$.
  Then, by (iv), there is some $H'\in \Lambda \hat{X}$, with
  $n^{-1}(H')=\bigcup_{i\in I} V_i$ and
  $H'\subseteq \bigcup_{i\in I} \Diamond V_i$. Now, using the equality
  $\alpha h(y) = \beta h(y)$, we have:
  \[
    1 = y_{_{\bigcup_{i\in I} V_i}} = y_{n^{-1}(H')}= \chi_{H'}
    ({(y_{_U})}_{U\in \Lambda X}).
  \]
  Consequently,
  \[
    {(y_{_U})}_{U\in \Lambda X} \in H' \subseteq \bigcup_{i\in I} \Diamond V_i.
  \]
  Thus, for some $j\in I$, ${(y_{_U})}_{U\in \Lambda X}\in \Diamond V_j$ that
  is, $y_{_{V_j}} = 1$, hence $V_j\in F_y$.
	
  \vskip1.5mm

  \noindent \textit{Proof of (B)}. It is clear that $B(x)$ is a filter of
  $(\Lambda X, \subseteq)$ with property ($\Diamond$).  Conversely, let $F$ be a
  filter of $(\Lambda X, \subseteq)$ with property ($\Diamond$), and put
  \[
    A = \{z\in X\mid B(z)\subseteq F\}.
  \]
  We show that $A$ is a non-empty irreducible closed set.

  Indeed, given $t\in X\setminus A$, there is some $V\in \Lambda X$ with
  $t\in V$ and $V \notin F$. But then all elements of $V$ belong to
  $X\setminus A$, thus $t\in V\subseteq X\setminus A$; hence, $A$ is closed. $A$
  is also non-empty, because, if for every $x\in X$, we have some
  $U_x\in \Lambda X$ with $U_x\notin F$ then, by ($\Diamond$), we obtain that
  $\bigcup_{x\in X}U_x = X\notin F$, which contradicts the fact that $F$ is a
  filter.
	
  To show that $A$ is irreducible, let $A = F_1 \bigcup F_2$ with $F_1$ and
  $F_2$ closed. If $A\neq F_1$ and $A\neq F_2$ then there is
  $x\in X\setminus F_1$ and $y\in X\setminus F_2$ with $x,y\in A$. But then we
  can find $U,V\in \Lambda X$ with $x\in U\subseteq X\setminus F_1$ and
  $y\in V\subseteq X\setminus F_2$, and $U\bigcap V \in F$. Taking into account
  that
  $U\bigcap V\subseteq (X\setminus F_1)\bigcap (X\setminus F_2) = X\setminus A,$
  then, for every $z\in U\bigcap V$, there is some $V_z\in B(z)$ with
  $V_z\subseteq U\bigcap V$ and $V_z\notin F$. But then
  $U\bigcap V =\bigcup_{z\in U\bigcap V}V_z$ belongs to $F$ with all
  $V_z\notin F$, which contradicts ($\Diamond$).

  Since $X$ is sober and $A\subseteq X$ is a non-empty irreducible closed set,
  we know that $A = \overline{\{x\}}$ for a unique $x\in X$. We show that
  $F=B(x)$. Clearly $B(x)\subseteq F$. Concerning the converse inclusion,
  condition ($\Diamond$) ensures that, for every $U\in F$, there is some
  $z\in U\bigcap A = U\bigcap \overline{\{x\}}$ -- otherwise, we would find
  $V_z\in \Lambda X$, with $z\in V_z\not\in F$ and $U=\bigcup_{z\in U}V_z$, a
  contradiction to ($\Diamond$); but then $x\in U$, i.e., $U\in B(x)$.
\end{proof}

\begin{corollary}\label{cor4.4}
  For the filter, the proper filter and the prime filter monads, the category
  $\catX_{\alg}$ is, respectively, $\ALat$, $\ADom$ and $\SPEC$. Moreover,
  $\catX_{\alg}$ is the closure under weighted limits of $\xdot$ in $\xupt$,
  thus, also in $\TOP_0$.
\end{corollary}

\begin{proof}
  The above theorem shows that, in all the three cases
  $\catA=\ALat, \, \ADom,\, \SPEC$, $\catA$ is indeed contained in
  $\catX_{\alg}$. On the other hand, since $\s\in\catA$, and $\catA$ is closed
  under weighted limits in $\xupt$, the diagram \eqref{eq4.2} is contained in
  $\catA$ whenever it is an equaliser diagram. Hence $\catA$ coincides with
  $\catX_{\alg}$. Moreover, every $X$ of $\xupt$ making diagram \eqref{eq4.2} an
  equaliser belongs to the closure under weighted limits of $\xdot$ in $\xupt$,
  because $\s$ belongs to $\xdot$. Indeed, for the open filter monad $\mT$, $\s$
  is homeomorphic to $TX$ with $X$ a singleton space, and, for the proper and
  the prime filter monad, $\s$ is homeomorphic to $T\s$. Therefore, in the three
  cases, $\catX_{\alg}$ is precisely the closure under weighted limits in
  $\xupt$ of $\xdot$; and also in $\TOP_0$, since $\xupt$ is closed under
  weighted limits in $\TOP_0$.
\end{proof}

\begin{corollary}\label{cor4.6}
  The categories $\ALat$, $\ADom$ and $\SPEC$ have weighted colimits.
\end{corollary}

\begin{proof}
  It is a consequence of Theorem~\ref{com->cocom}, Proposition~\ref{prop4.1} and
  Theorem~\ref{theo4.3}.
\end{proof}


\section{The idempotent split completion for the filter monad}
\label{sec:idemp-split-compl}

Let $\mF=\fmonad$ be the open filter monad on $\catX=\TOP_0$. As in the previous
section, we use $\leq$ to refer to the order induced in the hom-sets of $\TOP_0$
by the \emph{specialisation order}, thus the open filter monad is of {\KZ} type
with respect to $\geq$. Accordingly, in all notions and results of
Sections~\ref{sec:backgr-mater-kock} and \ref{sec:abstr-algebr-objects} on {\KZ}
monads, regarding adjunctions between morphisms, ``left adjoint'' interchanges
with ``right adjoint''.

As seen in Section~\ref{sec:abstr-algebr-objects}, the idempotent split
completion of $\catX_{\mF}$, denoted by $\kar(\catX_{\mF})$, is equivalent to
the full subcategory $\Spl(\catX^{\mF})$ of $\catX^{\mF}$. And
$\Spl(\catX^{\mF})$ consists of all $\mF$-algebras $(X,\alpha)$ for which there
is a morphism $t:X\to FX$ (in $\TOP_0$) such that $\ga \dashv t$. Moreover, it
is known that the subcategory $\catX_{\mF}$ is contained in $\ALat$
\cite{Esc98}, and the latter is closed under weighted limits in
$\CONTLAT$. Thus, we have the following full embeddings:
\[
  \catX_{\mF}\hookrightarrow \kar(\catX_{\mF})\hookrightarrow \ALat
  \hookrightarrow\catX^{\mF}=\CONTLAT.
\]

In this section we show that the idempotent split completion of $\catX_{\mF}$
consists precisely of all algebraic lattices whose set of compact elements forms
the dual of a frame.

\begin{notation}
  Along this section we use the symbol $K(X)$ to denote the set of compact
  elements of a directed complete poset (see Definition~\ref{def:domains}).
\end{notation}

\begin{remark}\label{rem5.1}
  Let $X$ and $Y$ be continuous lattices and let
  $\smash{\xymatrix{Y\ar@/^1ex/[r]^{\ga}&X\ar@/^1ex/[l]^{e}}}$ be in $\TOP_0$ with
  $\ga e=\id_X$ and $e\ga \leq \id_Y$. Then $\ga$ is defined by
  \[
    \ga(y)=\bigvee \{z\in X\mid e(z) \leq y\}.
  \]
  This follows from Freyd Adjoint Theorem.
\end{remark}

\vspace{-1.4em}
\begin{lemma}\label{lem5.2}
  Let $X,Y$ be directed complete posets with $Y$ continuous, and let
  $Y\adjunctop{t}{\ga}X$
  be in $\POSS$ with $\alpha$ a surjective map and $t$ preserving directed
  suprema. Then $\ga$ preserves the way-below relation $\ll$, and, as a
  consequence, $X$ is also continuous and the set of compact elements of $X$ is
  given by
  \[
    K(X)=\{\ga(y)\mid y\in K(Y)\}.
  \]
\end{lemma}

\begin{proof}
  Let $y_0, y_1 \in Y$ with $y_0\ll y_1$. Assume that
  $\ga(y_1)\leq \vdir_{i\in I} z_i$. Then, since $\ga \dashv t$,
  $y_1\leq t(\vdir_{i\in I} z_i)=\vdir_{i\in I} t(z_i)$. By hypothesis, there is
  some $i\in I$ with $y_0\leq t(z_i)$. Hence $\ga(y_0)\leq \ga t(z_i)\leq
  z_i$. Consequently, $\ga(y_0)\ll \ga(y_1)$. Thus $\ga$ prserves the relation
  $\ll$, in particular it preserves compact elements.

  Let now $x\in K(X)$. First we show that
  $x=\bigvee\{\ga(y)\mid y\in K(Y), \ga(y)\leq x\}$. Indeed, for every $y\in Y$,
  we have that the inequalities $\ga(y)\leq x$ and $y\leq t(x)$ are equivalent,
  because $\ga\dashv t$. Now, using also the fact that $\ga$ is surjective and
  $Y$ is continuous, we have that
  \begin{align*}
    x=\ga t(x)=\ga\left(\bigvee \{y\in K(Y)\mid y\leq t(x)\}\right)
    &=\bigvee \{\ga(y)\mid y\in K(Y),\, y\leq t(x)\}
    \\ &
    = \bigvee \{\ga(y)\mid y\in K(Y),\, \ga(y)\leq x\}.
  \end{align*}
  Let now $x\in K(X)$. The set $\{\ga(y)\mid y\in K(Y),\, \ga(y)\leq x\}$ is
  directed in $X$ (because it is the image under $\ga$ of a directed set). Then,
  as $x$ is compact, it must be of the form $\ga(y)$ for some $y\in K(Y)$.
\end{proof}

\begin{lemma}\label{lem5.3}
  Let $A$ be an algebraic lattice such that there is $t:A\to F A$ in $\TOP_0$
  which is right adjoint to the $\mF$-structure $\alpha: FA\to A$ (thus,
  $\alpha t=\id_A$ and $\id_{F A}\le t\alpha$). Then the set $K(A)$ of compact
  elements of $A$ is closed under arbitrary infima, and, in $K(A)$, finite
  suprema distribute over arbitrary infima.
\end{lemma}

\begin{proof}
  It is easy to see that in $F A$, the compact elements are closed under
  arbitrary infima. Indeed $K(F A)=\{\upc U\mid U\in \Omega A\}$, and we
  have that
  $\bigcap_{i\in I} \upc U_i = \upc \left(\bigcup_{i\in I} U_i\right)$.

  Moreover, in $K(F A)$ finite suprema are distributive with respect to
  arbitrary infima. Indeed, it is easy to see that, for $V_i$, $U$ and $V$ in
  $\goo A$, we have in $F A$:
  \begin{tfae}
  \item
    $\bigwedge_{i\in I}\upc {V}_i = \bigcap_{i\in I}\upc {V}_i =
    \upc \left(\bigcup_{i\in I}{V}_i\right)$; and
  \item
    $\left(\upc{U}\right)\bigvee \left(\upc{V}\right) = \upc
    \left({U}\bigcap {V}\right) $.
  \end{tfae}
  \noindent
  Hence,
  \begin{align*}
    \left(\upc {U}\right) \medvee \left(\bigwedge_{i\in I}\upc {V}_i\right)
       &= \upc \left({U}\medcap \left(\bigcup_{i\in I} {V}_i\right)\right) \\
       &= \upc \left(\bigcup_{i\in I} \left({U}\medcap{V}_i\right)\right)\\
       & = \bigcap_{i\in I} \upc \left({U} \medcap {V}_i\right)\\
       & = \bigwedge_{i\in I} \left(\left(\upc {U}\right)\medvee \left(\upc {V}_i \right)\right).
  \end{align*}
  Now, being simultaneously a right and a left adjoint, $\alpha$ preserves
  infima and suprema. Consequently, by Lemma \ref{lem5.2}, as in $FA$, compacts
  in $A$ are closed under infima. Moreover, $A$ also inherits the distribuivity
  of finite suprema over arbitrary infima for compact elements: putting
  $c=\ga(d)$ and $c_i=\ga(d_i)$ with $d$ and all $d_i$ compacts of $A$, we have:
  \begin{align*}
    c\vee \left(\bigwedge_{i\in I} c_i\right) & = \alpha \left(d\right) \vee\left(\bigwedge_{i\in I} \alpha \left(d_i\right)\right)= \alpha \left(d\vee \left(\bigwedge_{i\in I} d_i\right)\right) = \alpha \left(\bigwedge_{i\in I} \left(d\vee d_i\right)\right) 
                                                = \bigwedge_{i\in I} \left(c\vee c_i\right)
  \end{align*}
  \\[-1.2em]
\end{proof}

\begin{theorem} The idempotent split completion of the category $\catX_{\mF}$ of
  algebraic algebras is precisely the full subcategory of $\CONTLAT$ of all
  algebraic lattices whose subposet of compacts is the dual of a frame.
\end{theorem}

\begin{proof}
  We know that $\kar\left(\catX_{\mF}\right)$ consists of all algebraic lattices
  $A$ such that the $\mF$-structure of $A$, $\alpha:F A\to A$, has a right
  adjoint $t:A\to F A$.  In particular, since $\alpha$ is a retraction, also
  $\alpha t = \id_A$.  Consequently, by Lemma \ref{lem5.3}, for every
  $A\in \text{kar}\left(\catX_{\mF}\right)$, the poset dual to $K(A)$ is a
  frame.
	
  Conversely, let $A$ be an algebraic lattice such that in its subposet $K(A)$
  there are all infima and finite suprema are distributive with respect to
  arbitrary infima.
	
  By Remark \ref{rem5.1}, the $\mF$-structure map of $A$ is given by
  \[
    \alpha\left(\phi\right) = \bigvee \{x\in A\mid e_A(x) \subseteq
    \phi\},\;\;\; \phi\in FA.
  \]
  We show that $\ga$ has a right adjoint $t:A\ra FA$.

  For every $G\in \goo A$, let $k(G)$ denote the compact elements of A which
  belong to $G$. Given $a\in A$, consider the subset of $FA$
  \begin{equation}\label{sa}
    S_a=\{\phi \in FA\mid \ga(\phi)\leq a\}
  \end{equation}
  and the subset of $\goo A$
  \begin{equation}\label{psia}
    \psi_a=\{G\in \goo A\mid \bigwedge k(G)\leq a\}.
  \end{equation}
  We show that $\psi_a$ is a filter and $\psi_a=\bigvee S_a$.

  First, we show that the union of all filters of $S_a$ is precisely $\psi_a$.
  Let then $\phi\in FA$ with $\ga(\phi)\leq a$, and let $G\in \phi$. Put
  $c=\bigwedge k(G)$. By hypothesis, $c\in K(A)$, then $\upc c$ is an open
  set containing $G$, hence belongs to $\phi$. Consequently,
  $e_A(c)\subseteq \phi$, thus, $c\leq \ga(\phi)$. Since, $\ga(\phi)\leq a$, it
  follows that $c\leq a$, as desired.  Conversely, let $G$ be an open set of $A$
  with $ \bigwedge k(G)\leq a$. Put $\phi=\upc G$. Then, for every $x\in A$,
  $e_A(x)\subseteq \phi$ means that every open set of which $x$ is an element
  contains $G$, and, in particular, contains $k(G)$. But this implies that
  $x\leq c$ for all $c\in k(G)$, that is, $x\leq \bigwedge k(G)$, and, thus,
  $x\leq a$. Since this happens to all $x$ with $e_A(x)\subseteq \phi$, we have
  $\ga(\phi)\leq a$. Hence, $G\in \phi$ with $\phi$ a filter of $S_a$.

  Now, we show that $\psi_a$ is indeed a filter, then $\psi_a=\bigvee
  S_a$. First, observe that, every open $G$ is the union of all sets
  $\upc c$ with $c\in k(G)$, and, moreover, if
  $\{c_i,\, i\in I\}\subseteq K(A)$ with $G=\bigcup_{i\in I}\upc c_i$, then
  $\bigwedge k(G)=\bigwedge_{i\in I}c_i$. Now, let $G$ and $H$ belong to
  $\psi_a$ with $k(G)=\{c_i,\, i\in I\}$ and $k(H)=\{d_j,\, j\in J\}$. Then
  \[
    G\medcap H= \left(\bigcup_{i\in I}\upc c_i\right) \bigcap \left(
      \bigcup_{j\in J}\upc d_j\right)=\bigcup_{i\in I, \, j\in J}(\upc
    c_i \bigcap \upc d_j)=\bigcup_{i\in I,\, j\in J}\upc (c_i \vee d_j)
  \]
  with all $c_i\vee d_j$ compact, because the supremum of two compacts is
  compact. Moreover, using the existing distributivity in $K(A)$,
  \[
    \bigwedge_{i\in I, \, j\in J} (c_i \vee d_j)=\left(\bigwedge_{i\in
        I}c_i\right)\vee \left(\bigwedge_{j\in j}d_j\right)\leq a\wedge a=a.
  \]
  Then $G\medcap H$ belongs to $\psi_a$.

  Now, put, for every $a\in A$,
  \[
    t(a) =\bigvee S_a=\psi_a.
  \]
  By the defnition of $S_a$, $t:A\ra FA$ is indeed a right adjoint of $\ga$ in
  $\POSS$. It remains to show that the map $t$ is continuous (equivalently, it
  preserves directed suprema). We know that the sets
  $U^\#=\{\phi \in FA\mid U\in \phi\}$, $U\in \goo A$, form a base of the
  topology of $FA$ (see Examples \ref{exs:1}(3)). And we have that
  \[
    t^{-1}( U^\#)=\{a\in A\mid U \in t(a)\} =\{a\in A\mid\bigwedge k(U)\leq
    a\}= \big\uparrow (\bigwedge k(U)) ;
  \]
  thus $t^{-1}(U^\#)$ is open because, by hypothesis, $\bigwedge k(U)$ is
  compact.
\end{proof}



\end{document}